\documentclass[dvipdfmx,12pt]{amsart}
\usepackage{amsfonts}
\usepackage{amssymb}
\usepackage{amscd}
\usepackage{mathrsfs}
\usepackage{tikz}
\usepackage{hyperref}

\newtheorem{theorem}{Theorem}[section]

\newtheorem{lemma}[theorem]{Lemma}
\newtheorem{proposition}[theorem]{Proposition}

\newtheorem{remark}{Remark}[section]

\newtheorem*{theorem*}{Theorem}
\newtheorem*{corollary*}{Corollary}
\newtheorem*{conjecture*}{Conjecture}
\newtheorem*{lemma*}{Lemma}
\newtheorem*{proposition*}{Proposition}
\newtheorem*{problem*}{Problem}
\newtheorem*{axiom*}{Axiom}
\newtheorem*{remark*}{Remark}
\newtheorem*{example*}{Example}
\newtheorem*{exercise*}{Exercise}
\newtheorem*{definition*}{Definition}

\numberwithin{equation}{section}

\newcommand{\im}{{\rm Im}}

\renewcommand{\eqref}[1]{(\ref{#1})}

\renewcommand{\bigskip}{\vspace{0.2cm}}
\renewcommand{\l}{\left}
\renewcommand{\r}{\right}
\newcommand{\rank}{{\rm rank}}

\newcommand{\cleq}{\lesssim}

\newcommand{\norm}{\Vert}

\newcommand{\wto}{\rightharpoonup}

\def\eqref[#1]{(\ref{#1})}
\def\secref[#1]{Section~\ref{#1}}
\def\appref[#1]{Appendix~\ref{#1}}
\def\lemref[#1]{Lemma~\ref{#1}}
\def\thmref[#1]{Theorem~\ref{#1}}
\def\propref[#1]{Proposition~\ref{#1}}
\def\conjref[#1]{Conjecture~\ref{#1}}
\def\defref[#1]{Definition~\ref{#1}}
\def\abs[#1]{|#1|}
\def\norm[#1]{\left\Vert #1 \right\Vert}
\def\tbra[#1,#2]{\langle #1 | #2\rangle} 
\def\rbra[#1,#2]{\l( #1 , #2 \r)} 
\def\sbra[#1,#2]{[ #1 | #2 ]} 
\def\fbra[#1,#2]{\{ #1 | #2 \}} 
\def\besov[#1,#2,#3]{B_{#2,#3}^{#1}}
\def\hbesov[#1,#2,#3]{\dot{B}_{#2,#3}^{#1}}

\newcommand{\eps}{\varepsilon}

\newcommand{\la}{\lambda}

\newcommand{\del}{\partial}

\newcommand{\N}{{\mathbb N}}
\newcommand{\R}{{\mathbb R}}
\newcommand{\C}{{\mathbb C}}

\newcommand{\Z}{{\mathbb Z}}

\newcommand{\bS}{{\mathbf S}}

\newcommand{\cA}{{\mathcal A}}
\newcommand{\cB}{{\mathcal B}}

\newcommand{\cG}{{\mathcal G}}

\newcommand{\cI}{{\mathcal I}}

\newcommand{\cK}{{\mathcal K}}

\newcommand{\cP}{{\mathcal P}}
\newcommand{\cQ}{{\mathcal Q}}
\newcommand{\cR}{{\mathcal R}}

\newcommand{\scJ}{{\mathscr J}}
\newcommand{\scK}{{\mathscr K}}

\begin{document}

\title[Global Dynamics of sol.s with group invariance for NLS]{Global Dynamics of solutions with group invariance for the nonlinear Schr\"{o}dinger equation}
\author[T. Inui]{Takahisa Inui}
\address{Department of Mathematics\\ Graduate School of Science, Kyoto University, Kyoto\\ Kyoto, 606-8502, Japan}
\email{inui@math.kyoto-u.ac.jp}
\date{}
\keywords{global dynamics, orthogonal matrix, nonlinear Schr\"{o}dinger equation, group invariance}
\maketitle

\begin{abstract}
We consider the focusing mass-supercritical and energy-subcritical nonlinear Schr\"{o}dinger equation (NLS). We are interested in the global behavior of the solutions to (NLS) with group invariance. By the group invariance, we can determine the global behavior of the solutions above the ground state standing waves. 
\end{abstract}

\tableofcontents


\section{Introduction}


\subsection{Background}
We consider the focusing mass-supercritical and energy-subcritical nonlinear Schr\"{o}dinger equation:
\begin{equation}
\l\{
\begin{array}{ll}
i \del_t u +  \Delta u  + |u|^{p-1}u=0, & (t,x) \in \R \times \R^d,
\\
u(0,x)=u_0(x), & x \in \R^d,
\end{array}
\r.
\tag{NLS}
\label{NLS}
\end{equation}
where $d \in \N$ and $1+4/d<p<1+4/(d-2)$. Note that we regard $1+4/(d-2)$ as $\infty$ if $d=1,2$. 
It is known (see \cite{GV79} and the standard texts \cite{Caz03, Tao06, LP15}) that this equation \eqref[NLS] is locally well-posed in $H^1(\R^d)$ and the energy, the mass, and the momentum, which are defined as follows, are conserved. 
\begin{align}
\tag{Energy}
E(u)&:=\frac{1}{2}\norm[\nabla u]_{L^2}^2 - \frac{1}{p+1}\norm[u]_{L^{p+1}}^{p+1},
\\
\tag{Mass}
M(u)&:=\norm[u]_{L^2}^2,
\\
\tag{Momentum}
P(u)&:=\im \int_{\R^d}  \overline{u(x)} \nabla u(x)dx.
\end{align}

Since a pioneer work by Kenig and Merle \cite{KM06}, many researchers have studied the global dynamics for \eqref[NLS]. For the 3d cubic Schr\"{o}dinger equation, Holmer and Roudenko \cite{HR08} proved that the following two statements hold if the initial data $u_0 \in H^1$ is radially symmetric and satisfies the mass-energy condition $M(u_0)E(u_0) < M(Q) E(Q)$ where $Q$ is the ground state solution.
\begin{itemize}
\item $\norm[u_0]_{L^2}\norm[\nabla u_0]_{L^2}<\norm[Q]_{L^2}\norm[\nabla Q]_{L^2}$ $\Rightarrow$ the solution scatters.
\item $\norm[u_0]_{L^2}\norm[\nabla u_0]_{L^2}>\norm[Q]_{L^2}\norm[\nabla Q]_{L^2}$ $\Rightarrow$ the solution blows up in finite time.
\end{itemize} 
For the non-radial solutions, Duyckaerts, Holmer, and Roudenko \cite{DHR08} obtained the scattering result and Holmer and Roudenko \cite{HR10} proved that the solutions in the above blow-up region blow up in finite time or grow up at infinite time. Fang, Xie, and Cazenave \cite{FXC11} extended the scattering result and Akahori and Nawa \cite{AN13} extended both the scattering and the blow-up result to \eqref[NLS]. 
To explain their result, we introduce some notations. 
Let $\omega$ be a positive number. We define the action $S_\omega$ by
\begin{equation*}
S_{\omega}(\varphi):=E(\varphi)+\frac{\omega}{2} M(\varphi). 
\end{equation*} 
Moreover, let $K$ denote the functional which appears in the virial identity (see \eqref[1.1.1]), that is,
\begin{align*}
K(\varphi)
:=\del_\lambda (S_{\omega}(\varphi^\lambda))|_{\lambda=0}
=\frac{2}{d} \norm[\nabla \varphi]_{L^2}^2 - \frac{p-1}{p+1} \norm[\varphi]_{L^{p+1}}^{p+1},
\end{align*}
where $\varphi^{\lambda}(x):=e^{\lambda} \varphi(e^{\frac{2}{d}\lambda}x)$. 
We consider the minimizing problem
\begin{equation*}
l_{\omega}:=\inf \{S_{\omega}(\varphi): \varphi \in H^1\setminus \{0\} , K(\varphi)=0 \}.
\end{equation*}
It is known  that there exists a unique radial positive solution $Q_{\omega}$ of the elliptic equation
\begin{equation*} 
-\Delta \varphi + \omega \varphi -|\varphi|^{p-1}\varphi=0,
\end{equation*}
and it attains the minimizing problem $l_{\omega}$, that is, $l_{\omega}=S_{\omega}(Q_{\omega})$ and $K(Q_{\omega})=0$ hold (see \cite{Str77, BeLi83_1, BeLi83_2, Kwo89}).
Fang, Xie, and Cazenave \cite{FXC11} proved (1) and Akahori and Nawa \cite{AN13} proved both (1) and (2) in the following theorem. See also (1.26) in \cite{AN13} for the setting by the frequency $\omega$. 
\begin{theorem} \label{thm1.1.0}
Let $\omega>0$, $u_0 \in H^1(\R^d)$ satisfy $S_{\omega}(u_0)<l_{\omega}$, and $u$ be the solution of \eqref[NLS] with the initial data $u_0$.  Then, the following statements hold. 
\begin{enumerate}
\item If $K(u_0) \geq 0$, then the solution $u$ scatters.
\item If $K(u_0)<0$, then the solution $u$ blows up in finite time or grows up at infinite time. More precisely, one of the following four cases occurs.
\begin{enumerate}
\item $u$ blows up in finite time in both directions.
\item $u$ blows up in positive finite time and $u$ is global in the negative time direction and $\limsup_{t\to -\infty} \norm[\nabla u(t)]_{L^2}=\infty$. 
\item $u$ blows up in negative finite time and $u$ is global in the positive time direction and $\limsup_{t\to \infty} \norm[\nabla u(t)]_{L^2}=\infty$. 
\item $u$ is global in both time directions and $\limsup_{t\to \pm \infty} \norm[\nabla u(t)]_{L^2}=\infty$. 
\end{enumerate} 
\end{enumerate}
\end{theorem}

In the present paper, we weaken the condition $S_{\omega}<l_{\omega}$ in \thmref[thm1.1.0] by group invariance.

For other studies of global dynamics of dispersive equations, see \cite{DR10, NS12CVPDE, DR15} (other global dynamics of \eqref[NLS]),  \cite{KM06, DM09, LZ09, KV10, Dod14pre, NR15pre} (energy-critical NLS), \cite{KVZ08, KTV09, Dod15} (mass-critical NLS), \cite{Mas15, Mas16pre, KMMV16pre} (mass-subcritical NLS), \cite{KM08, DKM13, DJKM16pre} (wave equations), \cite{IMN11, KSV12, NS12ARMA, IMN14} (nonlinear Klein-Gordon equations), and references therein.


\subsection{Main Result}

We are interested in the global behavior of the solutions with group invariance for \eqref[NLS]. 
Let $O(d)$ denote the set of $d \times d$ orthogonal matrices, \textit{i.e.} 
\begin{equation*}
O(d):=\{ \cR \in M_d(\R): \cR^T \cR=\cI_d \}, 
\end{equation*}
where the transpose of a matrix $\cA$ is written by $\cA^T$ and $\cI_d$ denotes the $d \times d$ identity matrix. 
$\R/2\pi \Z \times O(d)$ is a group with the binary operation $(+,\cdot)$. Let a subgroup $G$ of $\R/2\pi \Z \times O(d)$ satisfy the following assumption throughout this paper. 

(A). For $(\theta_1,\cG_1), (\theta_2,\cG_2) \in G$, if $\cG_1=\cG_2$, then we have $\theta_1=\theta_2$. 

\noindent Due to the assumption (A), we can use the notation $\cG$  without confusion to denote not only a matrix but also an element of $G$. 
For a subgroup $G$ of $\R/2\pi\Z \times O(d)$, we say that a function $\varphi$ is $G$-invariant (or with $G$-invariance) if $ \varphi=\cG\varphi$ for all $\cG \in G$, where $\cG \varphi(x):= e^{-i\theta} (\varphi \circ \cG^{-1})(x)=e^{-i\theta} \varphi(\cG^{-1} x)$ for $\cG=(\theta,\cG)\in \R/2\pi\Z \times O(d)$. 
We define the Sobolev space with $G$-invariance by
\[ H_{G}^1:=\{ \varphi \in H^1(\R^d): \varphi=\cG \varphi, \forall \cG \in G\}. \]
If the initial data $u_0$ belongs to $H_{G}^1$, then the corresponding solution to \eqref[NLS] is also $G$-invariant since the Laplacian $\Delta$ is invariant for group actions by $\R/2\pi\Z \times O(d)$ and \eqref[NLS] is gauge invariant. 
We remark that if (A) does not hold, then the results in the present paper is trivial since we have $H_{G}^1=\{0\}$.

For a subgroup $G$ of $\R/2\pi\Z \times O(d)$, we consider the restricted minimizing problem
\begin{equation*}
l_{\omega}^G:=\inf \{S_{\omega}(\varphi): \varphi \in H_{G}^1\setminus \{0\} , K(\varphi)=0 \},
\end{equation*}
and we define subsets $\scK_{G,\omega}^{\pm}$ in $H^1(\R^d)$ by
\begin{align*}
\scK_{G,\omega}^{+}&:=\{ \varphi \in H_{G}^1: S_{\omega}(\varphi)<l_{\omega}^{G}, K(\varphi)\geq 0 \},
\\
\scK_{G,\omega}^{-}&:=\{ \varphi \in H_{G}^1: S_{\omega}(\varphi)<l_{\omega}^{G}, K(\varphi)<  0 \}.
\end{align*}

We define a critical action for the data with $G$-invariance by 
\begin{align*}
{\bf S}_{\omega}^{G}:=&\sup\{ {\bf S}\in \R: \forall \varphi \in \scK_{G,\omega}^{+}, S_{\omega}(\varphi)<{\bf S} 
\\
& \Rightarrow \text{ the solution to {\eqref[NLS]} with the initial data } \varphi \text{ belongs to } L^{\alpha}(\R:L^r(\R^d))\}.
\end{align*} 
See \eqref[3.0] below for the definition of $\alpha$ and $r$. We remark that $u \in L^{\alpha}(\R:L^r(\R^d))$ implies that the solution $u$ scatters (see \propref[prop3.2]). Here, we say that the solution $u$ to \eqref[NLS] scatters if and only if there exist $\varphi_{\pm} \in H^1(\R^d)$ such that 
\[ \norm[u(t)-e^{it\Delta} \varphi_{\pm}]_{H^1} \to 0 \text{ as } t \to \pm \infty,\]
where $e^{it\Delta} $ denotes the free propagator of the Schr\"{o}dinger equation. 
We say that a subgroup $G'$ of $G$ satisfies {\rm ($\ast$)} if there exists a sequence $\{x_n\} \subset \R^d$ such that 
\begin{equation*}
\l\{
\begin{array}{l}
\{x_n - \cG' x_n\} \text{ is bounded for all } \cG' \in G',
\\
|x_n - \cG x_n| \to \infty \text{ as } n \to \infty, \text{ for all } \cG \in G\setminus G'.
\end{array}
\r.
\end{equation*}
For a finite group $G$, we define 
\[ m_{\omega}^{G} := \min_{G' \subsetneq G  \text{ satisfying {\rm ($\ast$)}} } \frac{\#G}{\#G'} \bS_{\omega}^{G'}, \]
where $\#X$ denotes the number of the elements in a set $X$. 

Our aim in the present paper is to prove the following theorem.
\begin{theorem} \label{thm1.1}
Let $\omega>0$, $G$ be a subgroup of $\R/2\pi\Z \times O(d)$, $u_0 \in H_G^1$ satisfy $S_{\omega}(u_0)< l_{\omega}^G$, and $u$ be the solution of \eqref[NLS] with the initial data $u_0$.  Then, the following statements hold. 
\begin{enumerate}
\item
In addition, we assume either that (i) $G$ is a finite group and $u_0$ satisfies $S_{\omega}(u_0)< m_{\omega}^{G}$ or that (ii) $G$ is an infinite group such that the embedding $H_{G}^1 \hookrightarrow L^{p+1}(\R^d)$ is compact. If $K(u_0)\geq 0$, then the solution $u$ scatters.
\item If $K(u_0)<0$, then the solution $u$ blows up in finite time or grows up at infinite time. More precisely, one of (a)--(d) in \thmref[thm1.1.0] occurs
\end{enumerate}
\end{theorem}

If $G=\{(0,\cI_d)\}$, then \thmref[thm1.1] is nothing but \thmref[thm1.1.0] (\cite{FXC11,AN13}). \thmref[thm1.1] means that we can classify the solution whose mass-energy is larger than that of the ground state standing waves into scattering and blow-up (grow-up) by group invariance. 
See Appendix \ref{secB} for the applications of \thmref[thm1.1]. 

\begin{remark} 
\ 
\begin{enumerate}
\item For \thmref[thm1.1] (1):
\begin{enumerate}
\item[(i)] If $G$ is finite, then the embedding $H_{G}^1 \hookrightarrow L^{p+1}(\R^d)$ is  not compact. 
\item[(ii)] If $G$ satisfies that for all $x \in \R^d \setminus \{0\}$, $Gx:=\{\cG x: \cG \in G \}$ has infinitely many elements, then the embedding $H_{G}^1 \hookrightarrow L^{p+1}(\R^d)$ is compact (see \cite{BW95}). 
\item[(iii)] In the case that $G$ is an infinite group such that  the embedding $H_{G}^1 \hookrightarrow L^{p+1}(\R^d)$ is  not compact, the scattering result for the solutions with $G$-invariance is an open problem. 
 \end{enumerate}
\item For \thmref[thm1.1] (2):
\begin{enumerate}
\item[(i)] If the solution blows up in finite time, then $\norm[\nabla u(t)]_{L^2}$ diverges at the maximal existence time by the local well-posedness. 
\item[(ii)] If $G=\{(0,\cG):\cG \in O(d)\}$, that is, the solution is radially symmetric, then the solution blows up in finite time in both time directions. See \cite{OG91} and \cite[Theorem 1.1 (2)]{HR08}.
\item[(iii)] If the initial data $u_0$ satisfies $|x|u_0 \in L^2(\R^d)$, then the corresponding solution $u$ blows up in finite time in both time directions. This statement follows from Glassey's argument (see \cite{Gla77}) and the virial identity
\begin{equation}
\label{1.1.1}
\frac{d^2}{dt^2} \norm[xu(t)]_{L^2}^2 = 4dK(u(t)).
\end{equation}
See \cite{Mar97} for other sufficient conditions to blow up. 
\item[(iv)] For the cubic nonlinear Schr\"{o}dinger equation in two dimensions, Martel and Rapha\"{e}l \cite{MR15pre} obtained a grow-up solution. However, we don't know whether a grow-up solution exists or not for \eqref[NLS]. 
\end{enumerate}
\end{enumerate}
\end{remark}


\subsection{Idea of proof}

The proof of the scattering part, \thmref[thm1.1] (1), is based on the contradiction argument by Kenig and Merle \cite{KM06}. 
That is, we find a critical element, whose orbit is precompact in $H_{G}^1$, by assuming that \thmref[thm1.1] (1) fails and using  a concentration compactness argument, and we eliminate it by a rigidity argument. 
In the argument, we use the non-admissible Strichartz estimate, which was also used in Fang, Xie, and Cazenave \cite{FXC11}. 
However, unlike \cite{FXC11} and \cite{AN13}, we use a linear profile decomposition lemma for functions with group invariance to extend the mass-energy condition. See \propref[LPD] in the case that $G$ is finite. That is why we need to modify the construction of a critical element and the rigidity argument. If $G$ satisfies that the embedding $H_{G}^1 \hookrightarrow L^{p+1}(\R^d)$ is compact, we can eliminate the translation parameter in the linear profile decomposition by the compactness of the embedding. In this case, the same argument as in the radial case does work. 

The blow-up part, \thmref[thm1.1] (2), follows directly from the result of Du, Wu, and Zhang \cite{DWZ13pre}.

The rest of the present paper is organized as follows. In Section 2, we reorganize variational argument for the data with $G$-invariance and show the blow-up result, \thmref[thm1.1] (2), by the method of \cite{DWZ13pre}. 
Section 3 is devoted to preliminaries for the proof of the scattering result, \thmref[thm1.1] (1).
In Section 4, we consider the case that $G$ is finite. In Section 4.1, we show the linear profile decomposition lemma for functions with the finite group invariance, which is a key ingredient. In Section 4.2, we prove \thmref[thm1.1] (1) (i) by constructing a critical element and the rigidity argument. In Section 5, we prove \thmref[thm1.1] (1) (ii). Its proof is similar to in the radial case. 
We collect useful lemmas in Appendix A. In Appendix B, we introduce some applications of \thmref[thm1.1].


\section{Variational Argument and Blow-Up Result}


\begin{lemma} \label{lem2.2}
If $K(\varphi) \geq 0$, then we have
\begin{equation}
S_{\omega}(\varphi) \leq \frac{1}{2}\norm[\nabla \varphi]_{L^2}^2 +\frac{\omega}{2} \norm[\varphi]_{L^2}^2 
\leq \frac{d(p-1)}{d(p-1)-4} S_{\omega}(\varphi). 
\end{equation}
\end{lemma}

\begin{proof}
The left inequality holds obviusly. We prove the right inequality. We have
\begin{align*}
0 \leq K(\varphi)
= \l( \frac{2}{d}-\frac{p-1}{2}\r) \norm[\nabla \varphi]_{L^2}^2 +(p-1) E(\varphi).
\end{align*} 
Adding $\omega(p-1) M(\varphi )/2$, we obtain
\begin{equation}
\l( \frac{p-1}{2}-\frac{2}{d}\r) \norm[\nabla \varphi]_{L^2}^2 +\frac{\omega}{2}(p-1) M(\varphi ) \leq (p-1) S_{\omega}(\varphi).
\end{equation}
Therefore, 
\begin{equation}
\l( p-1-\frac{4}{d}\r) \l\{ \frac{1}{2} \norm[\nabla \varphi]_{L^2}^2 +\frac{\omega}{2} M(\varphi ) \r\} \leq (p-1) S_{\omega}(\varphi).
\end{equation}
This completes the proof. 
\end{proof}

\begin{lemma} \label{lem2.3}
If $u_0 \in \scK_{G,\omega}^{+}$, then the corresponding solution $u(t)$ belongs to $\scK_{G,\omega}^{+}$ for all existence time $t$. Moreover, if $u_0 \in \scK_{G,\omega}^{-}$, then the corresponding solution $u(t)$ belongs to $\scK_{G,\omega}^{-}$ for all existence time $t$.
\end{lemma}

\begin{proof}
Let $u_0 \in \scK_{G,\omega}^{+}$. Since the energy and the mass are conserved and the solution belongs to $H_{G}^1$, we have $u(t)\in \scK_{G,\omega}^{+} \cup \scK_{G,\omega}^{-}$ for all existence time $t$. We assume that there exists $t_{1}>0$ such that $u(t_{1})\in \scK_{G,\omega}^{-}$. By the continuity of the solution in $H^1(\R^d)$, there exists $t_{0}\in (0,t_{1})$ such that $K(u(t_{0}))=0$ and $K(u(t))<0$ for $t\in (t_{0},t_{1}]$. By the definition of $l_{\omega}^{G}$, if $u(t_{0})\neq 0$, then we see that
\begin{align*}
l_{\omega}^{G}
>E(u_0)+\frac{\omega}{2} M(u_0)
=E(u(t_{0}))+\frac{\omega}{2} M(u(t_{0}))
\geq l_{\omega}^{G}.
\end{align*}
This is a contradiction. Thus, $u(t_{0})=0$. By the uniqueness of the solution, $u=0$ for all time. However, this contradicts $u(t_{1})\in \scK_{G,\omega}^{-}$. Thus, we see that $u(t)\in \scK_{G,\omega}^{+}$ for all $t$. The second statement follows from the same argument.
\end{proof}

\begin{remark}
By Lemmas \ref{lem2.2} and \ref{lem2.3}, we obtain the global existence of the solution in $\scK_{G,\omega}^{+}.$
\end{remark}

\begin{lemma} \label{lem2.4}
Let $\varphi \in H_{G}^1$ satisfy $S_{\omega}(\varphi) < l_{\omega}^{G}$. Then, one of the following holds. 
\begin{equation}
K(\varphi) \geq \min \{  4(l_{\omega}^{G} - S_{\omega}(\varphi))/d ,\delta \norm[\nabla \varphi]_{L^2}^2 \}, \text{ or } K(\varphi) \leq -4(l_{\omega}^{G} - S_{\omega}(\varphi))/d,
\end{equation}
for some $\delta>0$. 
\end{lemma}

\begin{proof}
Since the statement holds if $\varphi = 0$, we may assume that $\varphi \neq 0$. Let $s(\la):=S_{\omega}(\varphi^\lambda)$, where $\varphi^\lambda(x)=e^{\lambda} \varphi(e^{\frac{2}{d}\lambda})$. Then, $s(0)=S_{\omega}(\varphi)$ and $s'(0)=K(\varphi)$. By direct calculations, 
\begin{align}
s(\lambda)
&= \frac{1}{2} e^{\frac{4}{d}\lambda} \norm[\nabla \varphi]_{L^2}^2 +\frac{\omega}{2} \norm[\varphi]_{L^2}^2 -\frac{1}{p+1} e^{(p-1)\lambda} \norm[\varphi]_{L^{p+1}}^{p+1},
\\
s'(\lambda)
&= \frac{2}{d} e^{\frac{4}{d}\lambda} \norm[\nabla \varphi]_{L^2}^2 -\frac{p-1}{p+1} e^{(p-1)\lambda} \norm[\varphi]_{L^{p+1}}^{p+1},
\\
s''(\lambda)
&= \frac{8}{d^2} e^{\frac{4}{d}\lambda} \norm[\nabla \varphi]_{L^2}^2 -\frac{(p-1)^2}{p+1} e^{(p-1)\lambda} \norm[\varphi]_{L^{p+1}}^{p+1}.
\end{align}
Thus, we have $s'' \leq 4s'/d$. First, we consider the case of $K<0$. Let $\la_0$ be defined by
\begin{equation}
\la_0 := \l( p-1-\frac{4}{d} \r)^{-1} \log \l( \frac{ \frac{2}{d}  \norm[\nabla \varphi]_{L^2}^2}{\frac{p-1}{p+1} \norm[\varphi]_{L^{p+1}}^{p+1}}\r).
\end{equation}
Then, $s'(\la_0)=0$. Moreover, $\la_0 < 0$ since $K<0$. Integrating $s'' \leq 4s'/d$ on $[\la_0, 0]$, we obtain 
\begin{equation}
s'(0) - s'(\la_0) \leq \frac{4}{d} (s(0)-s(\la_0)). 
\end{equation}
This completes the proof in the case of $K<0$. Next, we consider the case of $K >0$.  We define 
\begin{equation}
\la_1 := \l( p-1-\frac{4}{d} \r)^{-1} \log \l( \frac{ \frac{16}{d^2}  \norm[\nabla \varphi]_{L^2}^2}{\frac{p-1}{p+1} \l(p-1+\frac{4}{d} \r) \norm[\varphi]_{L^{p+1}}^{p+1}}\r).
\end{equation}
Then, $s''(\la_1)+4s'(\la_1)/d=0$. If $\la_1 \geq 0$, then, by the definition of $\la_1$, we obtain
\begin{equation}
K(\varphi) \geq \l(p-1+\frac{4}{d} \r)^{-1} \frac{2}{d} \l(p-1-\frac{4}{d}\r) \norm[\nabla \varphi]_{L^2}^2. 
\end{equation}
Letting $\delta:=2(p-1-4/d)/\{d(p-1+4/d )\}$, we obtain $K(\varphi) \leq \delta \norm[\nabla \varphi]_{L^2}^2$.
If $\la_1 < 0$, then $s''(\la) < -4s'(\la)/d$ for $\la \in [0, \la_0]$, where we note that $\la_0>0$ since $K \geq 0$ and that $s''(\la) + 4s'(\la)/d<0$ for all $\la > \la_1$.  Integrating the inequality $s''(\la) < -4s'(\la)/d$ on $[0,\la_0]$, this completes the proof. 
\end{proof}



\begin{proof}[Proof of {\thmref[thm1.1] (2)}]
By Lemmas \ref{lem2.3} and \ref{lem2.4}, if $u_0 \in \scK_{G,\omega}^{-}$, then the solution $u$ satisfies $K(u(t))< -4(l_{\omega}^{G}-S_{\omega}(u_0))/d<0$ for all existence time $t$. Therefore, \thmref[thm1.1] (2) follows directly from Theorem 2.1 in \cite{DWZ13pre}. 
\end{proof}


\section{Preliminaries for the Proof of the Scattering Result}
\label{sec3}

In this subsection, we introduce some basic facts used to prove the scattering result. Their proofs also can be found in \cite{FXC11}.
Let 
\begin{align} \label{3.0}
\begin{array}{lll}
\alpha:=\frac{2(p-1)(p+1)}{4-(d-2)(p-1)},
&
\beta:=\frac{2(p-1)(p+1)}{d(p-1)^2+(d-2)(p-1)-4},
& 
\\
 & & 
\\
q:=\frac{4(p+1)}{d(p-1)},
&
r:=p+1,
&
s:=\frac{d}{2}-\frac{2}{p-1}. 
\end{array}
\end{align}
Moreover, let $\beta'$ and $r'$ denote the H\"{o}lder exponents of the exponent $\beta$ and $r$, respectively.
\begin{lemma}[Strichartz estimates]
The following estimates are vaild.
\begin{align}
\label{3.1}
&\norm[e^{it\Delta} \varphi]_{L^q(\R:L^r)}
\cleq \norm[\varphi]_{L^2},
\\
\label{3.2}
&\norm[e^{it\Delta} \varphi]_{L^\alpha(\R:L^r)} \cleq \norm[\varphi]_{\dot{H}^s},
\\
\label{3.3}
&\norm[\int_{0}^{t} e^{i(t-t')\Delta} f(t') dt' ]_{L^\alpha(I:L^r)} \cleq \norm[f]_{L^{\beta'}(I:L^{r'})},
\end{align}
where $I$ is a time interval and the implicit constant is independent of $I$.  
\end{lemma}

\begin{proof}
The first estimate is a standard Strichartz estimate. The second one is obtained by the Sobolev inequality and the Strichartz estimate. The third one is a non-admissible Strichartz estimate (see \cite[Lemma 2.1]{CW92} or \cite[Proposition 2.4.1]{Caz03}). 
\end{proof}

\begin{proposition} \label{prop3.2}
Let $u_0 \in H^1(\R^d)$ and $u$ be the solution to \eqref[NLS] with the initial data $u_0$. If the solution $u$ is positively global and $u \in L^\alpha((0,\infty):L^r(\R^d))$, then the solution scatters in the positive time direction. Moreover, the same statement holds in the negative case. 
\end{proposition}

\begin{proof}
Proposition 2.3 in \cite{CW92} implies $u$ belongs to $L^{\eta}((0,\infty):L^{\rho}(\R^d))$ for any admissible pair $(\eta,\rho)$, and then a standard argument gives us the fact that $u$ scatters in the positive time directions (see the argument in Theorem 7.8.1 in \cite{Caz03}). 
\end{proof}

\begin{proposition}
\label{SDS}
There exists $\eps_{sd}>0$ such that if $u_0 \in H^1(\R^d)$ and $\norm[e^{it\Delta} u_0]_{L^{\alpha}((0,\infty):L^r)} \leq \eps_{sd}$, then the solution $u$ of \eqref[NLS] with the initial data $u_0$ is positively global and we have
\begin{equation}
 \norm[u]_{L^\alpha((0,\infty):L^r)}  \cleq \eps_{sd}. 
\end{equation}
In particular, if $\norm[u_0]_{H^1} \leq \eps_{sd}$, then the solution $u$  is global and we have
\begin{equation}
 \norm[u]_{L^\alpha(\R:L^r)}  \cleq \norm[u_0]_{H^1} . 
\end{equation}
\end{proposition}

\begin{proof}
The first statement follows form Proposition 2.4 in \cite{CW92}. 
By \eqref[3.2], we obtain $\norm[e^{it\Delta}u_0]_{L^\alpha(\R:L^r)} \cleq \norm[u_0]_{\dot{H}^s} \leq \eps_{sd}$. Applying  Proposition 2.4 in \cite{CW92}, we obtain the second statement.
\end{proof}

\begin{lemma}
\label{lem3.4}
If $\psi \in H_{G}^1$ satisfies $\norm[\nabla \psi]_{L^2}^2/2+  \omega  M(\psi)/2< l_{\omega}^G$, then there exists a global solution $U_{+}$ to \eqref[NLS] such that $U_{+}(0) \in \scK_{G,\omega}^{+}$ and $\norm[U_{+}(t) - e^{it\Delta}\psi]_{H^1} \to 0$ as $t \to \infty$. Moreover, the same statement holds in the negative case.
\end{lemma}

\begin{proof}
We may assume that $\psi \neq 0$ since the statement is true if $\psi=0$. 
It is known in \cite[Theorem 17]{Str81_2} (see also \cite[Theorem 8]{Str81_1}) that there exist $T \in \R$ and a unique solution $u \in C((T,\infty):H^1(\R^d) )$ of \eqref[NLS] such that 
\begin{equation}
\label{3.6}
\norm[U_{+}(t) - e^{it \Delta}\psi]_{H^1} \to 0 \text{ as } t \to \infty. 
\end{equation}
The uniqueness and the assumption that $\psi$ is $G$-invariant imply that the solution $U_{+}$ is also $G$-invariant. 
By the triangle inequality, the Sobolev embedding, \eqref[3.6], and $\norm[e^{it\Delta}\psi]_{L^{p+1}} \to 0$ as $t \to \infty$ (see \cite[Corollary 2.3.7]{Caz03}), we have
\begin{align*}
\norm[U_{+}(t)]_{L^{p+1}} 
&\leq \norm[U_{+}(t)-e^{it\Delta}\psi]_{L^{p+1}} +\norm[e^{it\Delta}\psi]_{L^{p+1}} 
\\
&\cleq  \norm[U_{+}(t)-e^{it\Delta}\psi]_{H^1} +\norm[e^{it\Delta}\psi]_{L^{p+1}}  \to 0,
\end{align*}
as $t \to \infty$. Therefore, by the conservation laws and the assumption, we obtain 
\begin{equation*}
S_{\omega}(U_{+})=\lim_{t \to \infty}S_{\omega}(U_{+}(t))=\frac{1}{2} \norm[\nabla \psi]_{L^2}^2+  \frac{\omega}{2}  M(\psi)< l_{\omega}^G
\end{equation*}
and 
\begin{equation*}
\lim_{t \to \infty}K(U_{+}(t))=\frac{2}{d}\norm[\nabla \psi]_{L^2}^2>0.
\end{equation*}
Thus, $U_{+}(t)$ belongs to $\scK_{G,\omega}^{+}$ for large $t>T$. This statement, Lemmas \ref{lem2.2}, and \ref{lem2.3}, imply that $U_{+}$ is global in both time directions and $U_{+}(0) \in \scK_{G,\omega}^{+}$. 
\end{proof}

\begin{lemma}[Perturbation Lemma]
\label{Perturb}
Given $A\geq 0$, there exist $\eps (A)>0$ and $C(A)>0$ with the following property. If $u \in C([0,\infty):H^1(\R^d))$ is a solution of \eqref[NLS], if $\tilde{u} \in C([0,\infty):H^1(\R^d))$ and $e\in L_{loc}^1([0,\infty): H^{-1}(\R^d))$ satisfy $i\del_t \tilde{u} + \Delta \tilde{u} +|\tilde{u}|^{p-1}\tilde{u}=e$, for a.e. $t>0$, and if 
\begin{align}
&\norm[\tilde{u}]_{L^\alpha([0,\infty):L^r)} \leq A, 
\\
&\norm[e]_{L^{\beta'}([0,\infty):L^{r'})} \leq \eps(A), 
\\
&\norm[e^{it\Delta} (u(0)- \tilde{u}(0))]_{L^\alpha([0,\infty):L^r)} \leq \eps  \leq \eps(A), 
\end{align}
then $u \in L^\alpha((0,\infty):L^r(\R^d))$ and $\norm[u- \tilde{u}]_{L^\alpha([0,\infty):L^r)} \leq C\eps$. 
\end{lemma}

See  \cite[Proposition 4.7]{FXC11} for the proof. 


\section{Proof of the Scattering Result for the finite group invariant solutions}

In this section, we prove \thmref[thm1.1] (1) (i). Let a subgroup $G$ of $\R/2\pi\Z \times \R^d$ be finite throughout this section. 


\subsection{Linear Profile Decomposition with finite group invariance}
We prove a linear profile decomposition for functions with $G$-invariance. 
Let $\tau_y \varphi(x)=\varphi(x-y)$ throughout this paper.  

\begin{proposition}[Linear Profile Decomposition with finite group invariance]
\label{LPD}
Let $\{\varphi_n\} \subset H_{G}^1$ be a bounded sequence. Then, after replacing a subsequence, for $j \in \N$ there exist a subgroup $G^j$ of $\R/2\pi\Z \times O(d)$, $\psi^j \in H_{G^j}^1$, $\{W_n^j\} \subset H_{G^j}^1$, $\{t_n^j\} \subset \R$, and $\{x_n^j\} \subset \R^d$ such that 
\begin{equation}
\varphi_n = \sum_{j=1}^J e^{i t_n^j \Delta} \sum_{\cG \in G}  \frac{\cG (\tau_{x_n^j} \psi^j )}{\#G} +  \sum_{\cG \in G}  \frac{\cG  W_n^J }{\#G} 
\end{equation}
for every $J\in \N$, and the following statements hold. 
\begin{enumerate}
\item For any fixed $j$, $\{t_n^j\}$ satisfies either $t_n^j=0$ or  $t_n^j \to \pm \infty$ as  $n \to \infty$,
\item For any fixed $j$, $\{x_n^j\}$ satisfies $x_n^j = \cG x_n^j$ for all $\cG \in G^j$ and $|x_n^j - \cG x_n^j| \to \infty$ for all $\cG \in G\setminus G^j$.
\item We have the orthogonality of the parameters: for $j\neq h$, 
\[\lim_{n \to \infty} |t_n^j- t_n^h|= \infty \text{ or }  \lim_{n \to \infty} |\cG x_n^j-\cG' x_n^h|=\infty \text{ for all } \cG,\cG' \in G.\]
\item We have smallness of the remainder: 
\[ \limsup_{n \to \infty} \norm[e^{it\Delta}  \sum_{\cG \in G}  \frac{\cG  W_n^J }{\#G}  ]_{L^\alpha(\R:L^r)} \to 0 \text{ as } J \to \infty.\] 
\item We have the orthogonality in norms: for all $\lambda \in [0,1]$,
\begin{align}
\norm[\varphi_n]_{\dot{H}^\lambda}^2 &= \sum_{j=1}^J \norm[\sum_{\cG \in G}  \frac{\cG (\tau_{x_n^j} \psi^j )}{\#G}]_{\dot{H}^\lambda}^2 +\norm[ \sum_{\cG \in G}  \frac{\cG  W_n^J }{\#G}  ]_{\dot{H}^\lambda}^2+o_n(1),
\\
\norm[\varphi_n]_{L^{p+1}}^{p+1} &= \sum_{j=1}^J \norm[e^{i t_n^j \Delta} \sum_{\cG \in G}  \frac{\cG (\tau_{x_n^j} \psi^j )}{\#G}]_{L^{p+1}}^{p+1} +\norm[ \sum_{\cG \in G}  \frac{\cG  W_n^J }{\#G} ]_{L^{p+1}}^{p+1}+o_n(1)
\end{align}
and, in particular, 
\begin{align} 
S_{\omega}(\varphi_n) &= \sum_{j=1}^J S_{\omega}\l(e^{i t_n^j \Delta} \sum_{\cG \in G}  \frac{\cG (\tau_{x_n^j} \psi^j )}{\#G} \r) +S_{\omega}\l( \sum_{\cG \in G}  \frac{\cG  W_n^J }{\#G}  \r)+o_n(1),
\\
K(\varphi_n) &= \sum_{j=1}^J K\l( e^{i t_n^j \Delta} \sum_{\cG \in G}  \frac{\cG (\tau_{x_n^j} \psi^j )}{\#G} \r) +K\l( \sum_{\cG \in G}  \frac{\cG  W_n^J }{\#G}  \r)+o_n(1).
\end{align}
\end{enumerate}
\end{proposition}

\begin{lemma}
\label{IS}
Let $a>0$ and $\{\varphi_n\} \subset H_G^1$ satisfy
\[ \limsup_{n \to \infty} \norm[\varphi_n]_{H^1} \leq a <\infty. \]
If $\norm[e^{it\Delta} \varphi_n ]_{L^\infty(\R:L^{p+1})} \to A$ as $n\to \infty$, then there exist a subsequence, which is still denoted by $\{\varphi_n\}_{n\in \N}$, a subgroup $G'$ of $G$, $\psi\in H_{G'}^1$, sequences $\{t_n\}_{n\in \N}\subset \R$, $\{x_n\}_{n\in \N} \subset \R^d$, and $\{W_n\}_{n\in \N} \subset H_{G'}^1$ such that 
\begin{equation} \label{1.1}
\varphi_n = e^{it_n \Delta}  \sum_{\cG \in G}  \frac{\cG (\tau_{x_n} \psi )}{\#G} +\sum_{\cG \in G}  \frac{\cG  W_n}{\#G} ,
\end{equation}
and the following hold. 
\begin{enumerate}
\item $e^{-it_n\Delta} \tau_{- \cG x_n} \varphi_n \wto \cG \psi/(\#G/\#G')$ in $H^1(\R^d)$ 
and $e^{-it_n\Delta} \tau_{-\cG x_n} \tilde{W}_n  \wto 0$ in $H^1(\R^d)$ for all $\cG \in G$, where $\tilde{W}_n:=\sum_{\cG \in G} \cG  W_n/\#G$.
\item The sequence $\{t_n\}$ satisfies either $t_n=0$ or  $t_n \to \pm \infty$ as  $n \to \infty$.
\item The sequence $\{x_n\}$ satisfies $\cG' x_n=x_n$ for all $\cG' \in G'$ and $|x_n - \cG x_n| \to \infty$ for all $\cG \in G\setminus G'$.
\item We have the orthogonality in norms:
\[ \norm[\varphi_n]_{\dot{H}^\lambda}^2 -  \norm[ \sum_{\cG \in G}  \frac{\cG (\tau_{x_n} \psi )}{\#G}]_{\dot{H}^\lambda}^2 -  \norm[\sum_{\cG \in G}  \frac{\cG  W_n}{\#G} ]_{\dot{H}^\lambda}^2 \to 0 \text{ as } n \to \infty,\]
for all $0 \leq \lambda \leq 1$.
\[ \norm[\varphi_n]_{L^{p+1}}^{p+1} -  \norm[ e^{it_n\Delta}  \sum_{\cG \in G}  \frac{\cG (\tau_{x_n} \psi )}{\#G}]_{L^{p+1}}^{p+1} -  \norm[\sum_{\cG \in G}  \frac{\cG  W_n }{\#G} ]_{L^{p+1}}^{p+1} \to 0 \text{ as } n \to \infty.\]
\item We have
\[ \norm[\psi]_{H^1} \geq \nu A^{\frac{d-2\Lambda^2}{2\Lambda(1-\Lambda)}} a^{-\frac{d-2\Lambda}{2\Lambda(1-\Lambda)}},\]
where $\Lambda:= d(p-1)/\{2(p+1)\} \in (0, \min\{1,d/2\})$ and the constant $\nu>0$ is independent of $a$, $A$, and $\{\varphi_n\}_{n\in \N}$.
\item If $A=0$, then for every sequences $\{t_n\}_{n\in \N}\subset \R$, $\{x_n\}_{n\in \N} \subset \R^d$, and $\{W_n\}_{n\in \N} \subset H^1(\R^d)$ satisfying \eqref[1.1] and (1), we must have $\psi=0$. 
\end{enumerate}
\end{lemma}

\begin{proof} 

Let $\widehat{\chi} \in C_0^{\infty}(\R^d)$ satisfy $0\leq \widehat{\chi} \leq 1$ and 
\begin{align*}
\widehat{\chi} (\xi)=\l\{ 
\begin{array}{ll}
1, & \text{ if } |\xi| \leq 1,
\\
0, & \text{ if } |\xi| \geq 2.
\end{array}
\r.
\end{align*}
Given $\rho>0$, we set $\widehat{\chi_{\rho}}(\xi):= \widehat{\chi}(\xi/\rho)$. 
For any $u \in H^1(\R^d)$  and $\lambda \in [0,1]$, we observe that
\begin{align*} 
\norm[u - \chi_\rho * u]_{\dot{H}^\lambda}^2
&=  \norm[ |\xi|^{\lambda}(\widehat{u} - \widehat{\chi_\rho} \widehat{ u} )]_{L^2}^2
=\int_{\R^d} |\xi|^{2\lambda} (1-\widehat{\chi_\rho})^2 |\widehat{u}|^2 d\xi 
\leq \int_{|\xi| \geq \rho} |\xi|^{2\lambda}  |\widehat{u}|^2 d\xi 
\\
&=\int_{|\xi| \geq \rho} |\xi|^{-2(1-\lambda)}  |\xi|^{2}  |\widehat{u}|^2 d\xi 
\leq \rho^{-2(1-\lambda)} \int_{|\xi| \geq \rho}  |\xi|^{2}  |\widehat{u}|^2 d\xi 
\\
& \leq \rho^{-2(1-\lambda)} \int_{\R^d}  |\xi|^{2}  |\widehat{u}|^2 d\xi  
= \rho^{-2(1-\lambda)} \norm[\nabla u]_{L^2}^2.
\end{align*}
Therefore, we have
\begin{equation} \label{1.2}
\norm[u - \chi_\rho * u]_{\dot{H}^\lambda} \leq \rho^{-(1-\lambda)} \norm[\nabla u]_{L^2}. 
\end{equation}
By Plancherel's theorem, we have
\[ \chi_{\rho}*u(x)= \int_{\R^d} \chi_\rho (y) u(x-y)dy = (-1)^{d} \int_{\R^d}  \widehat{\chi_\rho }(\xi) e^{-iy\cdot\xi} \widehat{u}(\xi)d\xi.\]
Since $\Lambda < d/2$, by the H\"{o}lder inequality,  we see that, for any $u \in H^1(\R^d)$,
\begin{align} 
\label{3.16}
|\chi_{\rho}*u(x)| 
&\leq  \int_{\R^d} |\widehat{\chi_\rho }(\xi) ||\widehat{u}(\xi)|d\xi  
\leq  \int_{|\xi| \leq 2\rho}  |\widehat{u}(\xi)|d\xi 
= \int_{|\xi| \leq 2\rho} |\xi|^{-\Lambda} |\xi|^{\Lambda}  |\widehat{u}(\xi)|d\xi
\\ \notag
&\leq \l( \int_{|\xi| \leq 2\rho} |\xi|^{-2\Lambda} d\xi \r)^{1/2} \norm[u]_{\dot{H}^{\Lambda}}\leq \kappa \rho^{\frac{d-2\Lambda}{2}} \norm[u]_{\dot{H}^{\Lambda}},  
\end{align}
where $\kappa$ is a  constant independent of $\rho$ and $u$. 

First, we consider the case of $A>0$. It follows from the Sobolev embedding $\norm[u]_{L^{p+1}} \leq C \norm[u]_{\dot{H}^\Lambda}$, the isometry of $e^{it\Delta}$ on $\dot{H}^{\Lambda}(\R^d)$,  \eqref[1.2], $0\leq \Lambda \leq 1$, and the assumption of $\limsup_{n\to\infty}\norm[\varphi_n]_{H^1} \leq a < \infty$ that 
\begin{align*} 
\norm[e^{it\Delta} \varphi_n -e^{it\Delta}(\chi_\rho * \varphi_n)]_{L^{p+1}} 
&\leq C\norm[e^{it\Delta} \varphi_n -e^{it\Delta}(\chi_\rho * \varphi_n)]_{\dot{H}^\Lambda} 
=C \norm[ \varphi_n -(\chi_\rho * \varphi_n)]_{\dot{H}^\Lambda} 
\\
& \leq  C \rho^{-(1-\Lambda)} \norm[\nabla \varphi_n]_{L^2}
\leq 2C \rho^{-(1-\Lambda)}  a,
\end{align*}
for large $n \in \N$. Choosing $\rho= (4C a/A)^{\frac{1}{1-\Lambda}}$, we obtain
\[ \norm[e^{it\Delta} \varphi_n -e^{it\Delta}(\chi_\rho * \varphi_n)]_{L^{p+1}} 
\leq 2C \rho^{-(1-\Lambda)}  a 
\leq \frac{A}{2}. \]
Thus, we have 
\begin{equation} 
\label{3.17}
\norm[e^{it\Delta} \varphi_n -e^{it\Delta}(\chi_\rho * \varphi_n)]_{L^{\infty}(\R:L^{p+1})} 
\leq \frac{A}{2}.
\end{equation}
By the triangle inequality, $\norm[e^{it\Delta} \varphi_n ]_{L^\infty(\R:L^{p+1})} \to A$ as $n\to \infty$, and \eqref[3.17], we get  
\begin{align}
\label{3.18}
&\norm[e^{it\Delta}( \chi_\rho *  \varphi_n )]_{L^\infty(\R:L^{p+1})}
\\ \notag
&\geq \norm[e^{it\Delta} \varphi_n ]_{L^\infty(\R:L^{p+1})} - \norm[e^{it\Delta} \varphi_n -e^{it\Delta}(\chi_\rho * \varphi_n)]_{L^{\infty}(\R:L^{p+1})} 
\\ \notag
& \geq \frac{3A}{4} -\frac{A}{2} = \frac{A}{4},
\end{align}
for large $n$. On the other hand, for large $n$, we have
\begin{align}
\label{3.19}
\norm[e^{it\Delta}( \chi_\rho *  \varphi_n )]_{L^\infty(\R:L^{p+1})} 
&\leq \norm[e^{it\Delta}( \chi_\rho *  \varphi_n )]_{L^\infty(\R:L^{2})}^{\frac{d-2\Lambda}{d}} 
\norm[e^{it\Delta}( \chi_\rho *  \varphi_n )]_{L^\infty(\R:L^{\infty})}^{\frac{2\Lambda}{d}}
\\ \notag
& \leq \norm[\varphi_n ]_{L^\infty(\R:L^{2})}^{\frac{d-2\Lambda}{d}} 
\norm[e^{it\Delta}( \chi_\rho *  \varphi_n )]_{L^\infty(\R:L^{\infty})}^{\frac{2\Lambda}{d}}
\\ \notag
& \leq (2a)^{\frac{d-2\Lambda}{d}} 
\norm[e^{it\Delta}( \chi_\rho *  \varphi_n )]_{L^\infty(\R:L^{\infty})}^{\frac{2\Lambda}{d}},
\end{align}
where we have used  the H\"{o}lder inequality, the isometry of $e^{it\Delta}$ on $L^2(\R^d)$, $\norm[\chi_\rho * u]_{L^2}\leq \norm[u]_{L^2}$ for all $u \in H^1(\R^d)$, and the assumption of $\limsup_{n\to\infty}\norm[\varphi_n]_{H^1} \leq a < \infty$.
Combining \eqref[3.18] and \eqref[3.19], we get, for large $n$,
\[\norm[e^{it\Delta}( \chi_\rho *  \varphi_n )]_{L^\infty(\R:L^{\infty})}^{\frac{2\Lambda}{d}} \geq (2a)^{-\frac{d-2\Lambda}{d}} \frac{A}{4},\]
or equivalently, 
\[\norm[e^{it\Delta}( \chi_\rho *  \varphi_n )]_{L^\infty(\R:L^{\infty})} 
\geq (2a)^{-\frac{d-2\Lambda}{2\Lambda}} \l( \frac{A}{4}\r)^{\frac{d}{2\Lambda}}. \]
Therefore, there exist $\{\tilde{t}_n\}_{n\in \N}\subset \R$ and $\{\tilde{x}_n\}_{n\in \N} \subset \R^d$ such that 
\begin{equation} 
\label{3.20}
\l|e^{i\tilde{t}_n\Delta}( \chi_\rho *  \varphi_n )(\tilde{x}_n)\r|
\geq (4a)^{-\frac{d-2\Lambda}{2\Lambda}} \l( \frac{A}{4}\r)^{\frac{d}{2\Lambda}},
\end{equation}
for large $n$.
We consider the following two cases. 
\begin{description}
\item[Case1] $\{\tilde{t}_n\}_{n\in \N}\subset \R$ is unbounded.
\item[Case2] $\{\tilde{t}_n\}_{n\in \N}\subset \R$ is bounded.
\end{description}

\noindent{\bf Case1:} Since $\{\tilde{t}_n\}_{n\in \N}\subset \R$ is unbounded, we may assume $\tilde{t}_n \to \pm \infty$ as $n \to \infty$ taking a subsequence. Let $t_n:=\tilde{t}_n$.
Taking a subsequence and using \lemref[lemA.0], we obtain a subgroup of $G'$ of $G$ such that a subsequence, which is still denoted by $\{\tilde{x}_n\}$, satisfies that
\begin{equation*}
\l\{
\begin{array}{ll}
\tilde{x}_n - \cG' \tilde{x}_n \to \bar{x}_{\cG'}  \text{ as } n \to \infty, & \forall \cG' \in G',
\\
|\tilde{x}_n - \cG \tilde{x}_n| \to \infty \text{ as } n \to \infty, & \forall \cG \in G\setminus G',
\end{array}
\r.
\end{equation*}
for some $\bar{x}_{\cG'} \in \R^d$. 
Using \lemref[lemA], we obtain a sequence $\{x_n\}$ such that 
\begin{equation*}
\l\{
\begin{array}{ll}
x_n - \cG' x_n =0, & \forall \cG' \in G',
\\
|x_n - \cG x_n| \to \infty \text{ as } n \to \infty, & \forall \cG \in G\setminus G',
\end{array}
\r.
\end{equation*}
and there exists $x_{\infty} \in \R^d$ such that
\[ x_n -\tilde{x}_n \to x_{\infty} \text{ as } n \to \infty. \]

Since $\norm[\varphi_n]_{H^1}$ is bounded, there exists $\psi \in H^1(\R^d)$ such that, after taking a subsequence, 
$e^{-it_n\Delta} \tau_{-x_n} \varphi_n \wto \psi/(\#G/\#G')$ in $H^1(\R^d)$ as $n \to \infty$. Here, we note that $\psi$ is $G'$-invariant since $\varphi_n$ is $G$-invariant and $x_n=\cG' x_n$ for all $\cG' \in G'$. We prove (5). Now, we have $w_n :=e^{-i\tilde{t}_n \Delta} \tau_{-\tilde{x}_n} \varphi_n \wto \tau_{x_\infty }  \psi/(\#G/\#G')$ in $H^1(\R^d)$ as $n \to \infty$. Since $e^{it\Delta}$ commutes with the convolution with $\chi_\rho$, we find that $e^{i\tilde{t}_n\Delta}( \chi_\rho *  \varphi_n )(\tilde{x}_n)= \chi_\rho *  w_n (0)$. By \eqref[3.16] and  \eqref[3.20], we have
\[
(4a)^{-\frac{d-2\Lambda}{2\Lambda}} \l( \frac{A}{4}\r)^{\frac{d}{2\Lambda}} 
\leq \l| \frac{\chi_\rho * \psi(-x_\infty)}{\#G/\#G'}\r|
\leq \kappa \rho^{\frac{d-2\Lambda}{2}}  \frac{\norm[ \psi]_{\dot{H}^{\Lambda}} }{\#G/\#G'}
\leq \kappa \rho^{\frac{d-2\Lambda}{2}}  \frac{\norm[\psi]_{H^{1}}}{\#G/\#G'}.
\]
Since we take $\rho= (4C a/A)^{\frac{1}{1-\Lambda}}$, we obtain the statement (5). We set $W_n :=\varphi_n -e^{it_n \Delta} \tau_{x_n} \psi$. Since $\varphi_n$ is $G$-invariant, we see that
\[ \varphi_n 
= \sum_{\cG \in G} \frac{ \cG \varphi_n}{\#G}
=\sum_{\cG \in G} \frac{ \cG (e^{it_n \Delta}  \tau_{x_n}\psi+W_n)}{\#G}
=\sum_{\cG \in G} \frac{ e^{it_n \Delta} \cG ( \tau_{x_n} \psi)}{\#G}
+\sum_{\cG \in G} \frac{ \cG W_n}{\#G}. \]
This is the statement \eqref[1.1]. Moreover, $W_n$ is $G'$-invariant since $\varphi_n$ and $\tau_{x_n}\psi$ are $G'$-invariant. 
We check the statement (1). The first statement $e^{-it_n\Delta} \tau_{- \cG x_n} \varphi_n \wto \cG \psi/(\#G/\#G')$ in $H^1(\R^d)$ follows from the definition of $\psi$ and the $G$-invariance of $\varphi_n$. We prove the second statement $e^{-it_n\Delta} \tau_{-\cG x_n} \tilde{W}_n  \wto 0$ in $H^1(\R^d)$ for all $\cG \in G$, where we recall that $\tilde{W_n}=\sum_{\cG \in G} \cG W_n/\#G$. 
Let $\{\cG_{k}\}_{k=1}^{\#G/\#G'}$ be the set of left coset representatives, that is, we have
\[ G= \sum_{k=1}^{\#G/\#G'} \cG_{k} G'. \] 
Since $W_n$ is $G'$-invariant, we find that
\[ \tilde{W_n}=\sum_{\cG \in G} \frac{\cG W_n}{\#G} = \sum_{k=1}^{\#G/\#G'} \frac{\cG_{k} W_n}{\#G/\#G'}.\]
Let $\cG=\cG_{l} \cG'$ for some  $l\in \{1,2,\cdots,\#G/\#G'\}$ and $\cG' \in G'$. Then, by the definition of $W_n$ and the first statement in (1), we obtain 
\begin{align*}
e^{-it_n\Delta} \tau_{-\cG x_n}  \tilde{W_n} 
&= e^{-it_n\Delta} \tau_{-\cG_{l} x_n} \varphi_n -  \sum_{k=1}^{\#G/\#G'} \frac{  \tau_{-\cG_{l} x_n +\cG_{k} x_n} \cG_{k} \psi}{\#G/\#G'}
\\
& \wto \frac{\cG_{l} \psi}{\#G/\#G'}- \frac{\cG_{l}\psi}{\#G/\#G'}=0,
\end{align*}
where we note that $|-\cG_{l} x_n +\cG_{k} x_n| =|-x_n +\cG_{l}^{-1} \cG_{k} x_n| \to \infty$ since $\cG_{l}^{-1} \cG_{k} \not\in G'$ if $k\neq l$. Thus, we get the second statement in (1). Next, we prove (4).
We set $\tilde{\psi}_n:=\sum_{\cG \in G} e^{it_n \Delta} \cG ( \tau_{x_n} \psi )/\#G$. 
We have
\begin{align*}
\norm[\varphi_n]_{\dot{H}^{\lambda}}^2
&=\norm[\tilde{\psi}_n+\tilde{W_n}]_{\dot{H}^{\lambda}}^2
=\norm[\tilde{\psi}_n]_{\dot{H}^{\lambda}}^2+\norm[\tilde{W_n}]_{\dot{H}^{\lambda}}^2 +2 \rbra[\tilde{\psi}_n,\tilde{W_n}]_{\dot{H}^{\lambda}}
\\
&=\norm[\psi]_{\dot{H}^{\lambda}}^2+\norm[\tilde{W_n}]_{\dot{H}^{\lambda}}^2 +2 \rbra[\tilde{\psi}_n,\varphi_n-\tilde{\psi}_n]_{\dot{H}^{\lambda}},
\end{align*}
where $\rbra[\cdot,\cdot]_{\dot{H}^{\lambda}}$ denotes the inner product in $\dot{H}^{\lambda}$.
We calculate $\rbra[\tilde{\psi}_n,\varphi_n-\tilde{\psi}_n]_{\dot{H}^{\lambda}}$. Since $\tau_{x_n} \psi$ is $G'$-invariant, we observe that 
\begin{align*}
\tilde{\psi}_n
= \sum_{\cG \in G} \frac{  e^{it_n \Delta}  \cG( \tau_{ x_n} \psi)}{\#G}
= \sum_{k=1}^{\#G/\#G'} \frac{  e^{it_n \Delta}  \cG_{k}( \tau_{ x_n} \psi)}{(\#G/\#G')}.
\end{align*}
By this observation, we have
\begin{align*}
&\rbra[\tilde{\psi}_n,\varphi_n-\tilde{\psi}_n]_{\dot{H}^{\lambda}}
\\
& = \rbra[\sum_{k=1}^{\#G/\#G'} \frac{  e^{it_n \Delta}  \cG_{k} ( \tau_{ x_n} \psi)}{\#G/\#G'}, \varphi_n-\sum_{l=1}^{\#G/\#G'}  \frac{ e^{it_n \Delta} \cG_{l}(\tau_{x_n} \psi)}{\#G/\#G'}]_{\dot{H}^{\lambda}}
\\
& = \frac{1}{(\#G/\#G')^2} \sum_{k=1}^{\#G/\#G'} \sum_{l=1}^{\#G/\#G'}  \rbra[  e^{it_n \Delta}  \cG_k( \tau_{ x_n} \psi),\varphi_n- e^{it_n \Delta}  \cG_l( \tau_{ x_n} \psi)]_{\dot{H}^{\lambda}}
\\
& = \frac{1}{(\#G/\#G')^2} \sum_{k,l=1}^{\#G/\#G'}  \l\{  \rbra[  e^{it_n \Delta}  \cG_k( \tau_{ x_n} \psi),\varphi_n]_{\dot{H}^{\lambda}} 
- \rbra[  e^{it_n \Delta}  \cG_k( \tau_{ x_n} \psi),e^{it_n \Delta}  \cG_l( \tau_{ x_n} \psi)]_{\dot{H}^{\lambda}} \r\}
\end{align*}

For the first term, we find that, for all $k \in \{1,2,\cdots,\#G/\#G'\}$,
\begin{equation}
\label{3.21}
\rbra[  e^{it_n \Delta}  \cG_k( \tau_{ x_n} \psi),\varphi_n]_{\dot{H}^{\lambda}}
= \rbra[  \psi,e^{-it_n \Delta}  \tau_{-x_n} \cG_k^{-1} \varphi_n]_{\dot{H}^{\lambda}}
\to \frac{ \norm[\psi]_{\dot{H}^{\lambda}}^2}{(\#G/\#G')}
\end{equation} 
since $\varphi_n$ is $G$-invariant and $e^{-it_n \Delta} \tau_{-x_n} \varphi_n$ weakly converges to $ \psi/(\#G/\#G')$ as $n \to \infty$ in $H^1(\R^d)$. 
For the second term, we obtain
\begin{align}
\label{3.22}
\rbra[  e^{it_n \Delta}  \cG_k( \tau_{ x_n} \psi),e^{it_n \Delta}  \cG_l( \tau_{ x_n} \psi)]_{\dot{H}^{\lambda}}
&=\rbra[ \psi,  \tau_{-x_n} \cG_k^{-1} \cG_l (\tau_{x_n} \psi )]_{\dot{H}^{\lambda}}
\\ \notag
&=\rbra[ \psi,  \tau_{-x_n +\cG_k^{-1} \cG_l x_n} \cG_k^{-1} \cG_l \psi )]_{\dot{H}^{\lambda}}
\\ \notag
&\to 
\l\{
\begin{array}{ll}
\norm[\psi]_{\dot{H}^{\lambda}}^2, & \text{if } k=l,
\\
0, & \text{if } k\neq l.
\end{array}
\r.
\end{align}
Combining \eqref[3.21] with \eqref[3.22], we get
\begin{align*}
\sum_{k,l=1}^{\#G/\#G'} & \l\{  \rbra[  e^{it_n \Delta}  \cG_k( \tau_{ x_n} \psi),\varphi_n]_{\dot{H}^{\lambda}} 
- \rbra[  e^{it_n \Delta}  \cG_k( \tau_{ x_n} \psi),e^{it_n \Delta}  \cG_l( \tau_{ x_n} \psi)]_{\dot{H}^{\lambda}} \r\}
\\
& \to  \sum_{k,l=1}^{\#G/\#G'}  \frac{ \norm[\psi]_{\dot{H}^{\lambda}}^2}{(\#G/\#G')}
- \sum_{k=1}^{\#G/\#G'}\norm[\psi]_{\dot{H}^{\lambda}}^2=0.
\end{align*}
This implies the first statement of (4).
We set
\[ f_n:= \l|\norm[\varphi_n]_{L^{p+1}}^{p+1}-\norm[\tilde{\psi}_n]_{L^{p+1}}^{p+1}-\norm[\tilde{W_n}]_{L^{p+1}}^{p+1}\r|. \]
We recall that, for every $P>1$ and $l\geq 2$, there exists a constant $C=C_{P,l}$ such that
\[ \l| \l|  \sum_{j=1}^l z_j \r|^{P} - \sum_{j=1}^l |z_j|^P \r| \leq C \sum_{j\neq k} |z_j||z_k|^{P-1},\]
for all $z_j \in \C$ for $1\leq j \leq l$. This implies that 
\[ ||z_1+z_2|^{p+1} -|z_1|^{p+1}- |z_2|^{p+1}| \leq C |z_1||z_2|(|z_1|^{p-1}+|z_2|^{p-1}). \]
Letting $g_n=|\tilde{\psi}_n|^{p-1}+|\tilde{W_n}|^{p-1}$, we get 
\begin{align*} 
f_n 
&\leq C \int_{\R^d} \l|\tilde{\psi}_n (x)\r| \l|\tilde{W_n}(x)\r| g_n(x) dx
\\
&\leq C \int_{\R^d} \l| \sum_{k=1}^{\#G/\#G'}  \frac{ e^{it_n\Delta} \cG_k( \tau_{ x_n} \psi)(x)}{\#G/\#G'}\r| \l|\tilde{W_n}(x)\r| g_n(x) dx
\\
&\leq C \sum_{k=1}^{\#G/\#G'} \int_{\R^d} \l| e^{it_n\Delta} \cG_k( \tau_{ x_n} \psi )(x)\r| \l|\tilde{W_n}(x)\r| g_n(x) dx
\\
&\leq C \sum_{k=1}^{\#G/\#G'} \int_{\R^d} \l| e^{it_n\Delta} \psi(x)\r| \l| \tau_{-x_n}  \cG_k^{-1} \tilde{W_n}(x)\r| \tau_{-x_n}  \cG_k^{-1} g_n(x) dx.
\end{align*}
Note that, by the triangle inequality and the Sobolev embedding, 
\begin{align*} 
\norm[\tau_{-x_n}  \cG_k^{-1} g_n ]_{L^{\frac{p+1}{p-1}}} 
&=\norm[g_n]_{L^{\frac{p+1}{p-1}}} 
\leq \norm[|\tilde{\psi}_n|^{p-1}]_{L^{\frac{p+1}{p-1}}} +\norm[|\tilde{W_n}|^{p-1}]_{L^{\frac{p+1}{p-1}}} 
\\
&= \norm[\tilde{\psi}_n]_{L^{p+1}}^{p-1} +\norm[\tilde{W_n}]_{L^{p+1}}^{p-1}
\cleq  \norm[\tilde{\psi}_n]_{H^1}^{p-1} +\norm[\tilde{W_n}]_{H^1}^{p-1}
\\
&\cleq  \norm[\psi]_{H^1}^{p-1} +\norm[W_n]_{H^1}^{p-1}
<C
 \end{align*}
where we use $\{W_n\}$ is bounded in $H^1$ since $\{\varphi_n\}$ is bounded. And $\norm[ \tau_{-x_n}  \cG_k^{-1} \tilde{W_n}]_{L^{p+1}}=\norm[\tilde{W_n}]_{L^{p+1}} < C$. Now, $\norm[e^{it_n \Delta}\psi]_{L^{p+1}}\to 0$ as $n \to \infty$ since $t_n \to \pm \infty$ (see \cite[Corollary 2.3.7 ]{Caz03}). Therefore, by the H\"{o}lder inequality, we get
\begin{align*}
f_n \cleq \norm[e^{it_n \Delta}\psi]_{L^{p+1}} \norm[ \tau_{-x_n}  \cG_k^{-1} \tilde{W_n}]_{L^{p+1}}\norm[ \tau_{-x_n}  \cG_k^{-1} g_n]_{L^{\frac{p+1}{p-1}}} \cleq \norm[e^{it_n\Delta}\psi]_{L^{p+1}} \to 0.
\end{align*}
This means the second statement of (4).

If $A=0$, then $\norm[e^{-it_n\Delta} \tau_{-x_n} \varphi_n]_{L^{p+1}}=\norm[e^{-it_n\Delta}  \varphi_n]_{L^{p+1}}\leq \norm[e^{-it\Delta}\varphi_n]_{L^\infty(\R: L^{p+1})} \to 0$. On the other hand, if $e^{-it_n\Delta} \tau_{-x_n} \varphi_n \wto \psi/(\#G/\#G')$ as $n \to \infty$ in $H^{1}(\R^d)$, then $e^{-it_n\Delta}  \tau_{-x_n} \varphi_n \to \psi/(\#G/\#G')$ as $n \to \infty$ in $L^{p+1}(B_R)$ for any $R>0$ by a compactness argument. Combining them, we get $\psi=0$.

\noindent{\bf Case2:} Since $\{\tilde{t}_n\}_{n\in \N}\subset \R$ is bounded, we may assume $\tilde{t}_n \to \bar{t} \in \R$ as $n \to \infty$ taking a subsequence. Let $t_n:=0$ for all $n$.  Minor modifications imply the statements, (1)--(3) and the first statement of (4). See the argument below (5.22) in \cite{FXC11}  for the second statement of (4) .
\end{proof}

\begin{proof}[Proof of {\propref[LPD]}]
We use an induction argument. By the boundedness of $\{\varphi_n\}$ in $H^1(\R^d)$, let $a:=\limsup_{n \to \infty} \norm[\varphi_n]_{H^1}<\infty$. Moreover, we define $A_1:=\limsup_{n \to \infty} \norm[e^{it \Delta}\varphi_n]_{L^{\infty}(\R:L^{p+1})}$. Note that the Sobolev embedding and the boundedness of $\{\varphi_n\}$ in $H^1(\R^d)$ give $A_1<\infty$. Taking a subsequence, we may assume that $A_1=\lim_{n \to \infty} \norm[e^{it \Delta}\varphi_n]_{L^{\infty}(\R:L^{p+1})}$. By \lemref[IS], we obtain a subsequence, which is still denoted by $\{\varphi_n\}_{n\in \N}$, a subgroup $G^1$ of $G$, $\psi^1 \in H_{G^1}^1$, sequences $\{t_n^1\}_{n\in \N}\subset \R$, $\{x_n^1\}_{n\in \N} \subset \R^d$, and $\{W_n^1\}_{n\in \N} \subset H_{G^1}^1$ such that 
\begin{equation*}
\varphi_n =  \tilde{\psi}_n^1  +\tilde{W}_n^1,
\end{equation*}
where $\tilde{\psi}_n^1:= e^{it_n^1 \Delta} \sum_{\cG \in G}  \cG (\tau_{x_n^1} \psi^1 )/\#G $ and $\tilde{W}_n^1 :=\sum_{\cG \in G}  \cG  W_n^1/\#G$ and the following statements hold. 
\begin{align*}
&e^{-it_n^1\Delta} \tau_{-\cG x_n^1} \varphi_n \wto \frac{\cG \psi^1}{\#G/\#G^1}, \text{ and } e^{-it_n^1\Delta} \tau_{-\cG x_n^1} \tilde{W}_n^1 \wto 0 \text{ in } H^1(\R^d) \text{ for all } \cG \in G.
\\
&t_n^1=0 \text{ or } t_n^1 \to \pm \infty \text{ as } n \to \infty.
\\
&\cG^1 x_n^1=x_n^1 \text{ for all } \cG^1 \in G^1 \text{ and } |x_n^1 - \cG x_n^1| \to \infty \text{ for all }\cG \in G\setminus G^1.
\\
&\norm[\varphi_n]_{\dot{H}^\lambda}^2 -  \norm[ \tilde{\psi}_n^1]_{\dot{H}^\lambda}^2 -  \norm[\tilde{W}_n^1]_{\dot{H}^\lambda}^2 \to 0 \text{ as } n \to \infty,
\text{ for all }0 \leq \lambda \leq 1.
\\ 
&\norm[\varphi_n]_{L^{p+1}}^{p+1} -  \norm[  \tilde{\psi}_n^1]_{L^{p+1}}^{p+1} -  \norm[\tilde{W}_n^1 ]_{L^{p+1}}^{p+1} \to 0 \text{ as } n \to \infty.
\\
&\norm[\psi^1]_{H^1} \geq \nu A_1^{\frac{d-2\Lambda^2}{2\Lambda(1-\Lambda)}} \l( \limsup_{n \to \infty} \norm[\varphi_n]_{H^1}\r)^{-\frac{d-2\Lambda}{2\Lambda(1-\Lambda)}},
\end{align*}
where $\nu>0$ is the constant appearing in \lemref[IS]. We call these properties $1$st properties. We note that $\tilde{W}_n^1$ is $G$-invariant and $A_2:=\limsup_{n\to \infty}\norm[\tilde{W}_n^1]_{L^\infty(\R:L^{p+1})}<\infty$. Taking a subsequence, we may assume that $A_2=\lim_{n\to \infty}\norm[\tilde{W}_n^1]_{L^\infty(\R:L^{p+1})}$. 
Applying \lemref[IS] as $\varphi_n=\tilde{W}_n^1$, we obtain a subsequence, which is still denoted by $\{\tilde{W}_n^1\}$, a subgroup $G^2$ of $G$, $\psi^2 \in H_{G^2}^1$, sequences $\{t_n^2\}_{n\in \N}\subset \R$, $\{x_n^2\}_{n\in \N} \subset \R^d$, and $\{W_n^2\}_{n\in \N} \subset H_{G^2}^1$ such that 
\begin{equation*}
\tilde{W}_n^1 = \tilde{\psi}_n^2  +\tilde{W}_n^2,
\end{equation*}
where $\tilde{\psi}_n^2:= e^{it_n^2 \Delta} \sum_{\cG \in G}  \cG (\tau_{x_n^2} \psi^2 )/\#G $ and $\tilde{W}_n^2 :=\sum_{\cG \in G}  \cG  W_n^2/\#G$ and the following statements hold. 
\begin{align*}
&e^{-it_n^2\Delta} \tau_{-\cG  x_n^2} \tilde{W}_n^1 \wto \frac{ \cG \psi^2}{\#G/\#G^2}, \text{ and } e^{-it_n^2\Delta} \tau_{-\cG x_n^2} \tilde{W}_n^2  \wto 0 \text{ in } H^1(\R^d) \text{ for all } \cG \in G.
\\
&t_n^2=0 \text{ or } t_n^2 \to \pm \infty \text{ as } n \to \infty.
\\
&\cG^2 x_n^2=x_n^2 \text{ for all } \cG^2 \in G^2 \text{ and } |x_n^2 - \cG x_n^2| \to \infty \text{ for all }\cG \in G\setminus G^2.
\\
&\norm[\tilde{W}_n^1]_{\dot{H}^\lambda}^2 -  \norm[ \tilde{\psi}_n^2]_{\dot{H}^\lambda}^2 -  \norm[\tilde{W}_n^2]_{\dot{H}^\lambda}^2 \to 0 \text{ as } n \to \infty,
\text{ for all }0 \leq \lambda \leq 1.
\\ 
&\norm[\tilde{W}_n^1]_{L^{p+1}}^{p+1} -  \norm[  \tilde{\psi}_n^2]_{L^{p+1}}^{p+1} -  \norm[\tilde{W}_n^2 ]_{L^{p+1}}^{p+1} \to 0 \text{ as } n \to \infty.
\\
&\norm[\psi^2]_{H^1} \geq \nu A_2^{\frac{d-2\Lambda^2}{2\Lambda(1-\Lambda)}} \l( \limsup_{n \to \infty} \norm[\tilde{W}_n^1]_{H^1}\r)^{-\frac{d-2\Lambda}{2\Lambda(1-\Lambda)}},
\end{align*}
where $\nu>0$ is the constant appearing in \lemref[IS]. Notice that the $1$st properties hold for the subsequence. Repeating this procedure, we obtain a subsequence, which is still denoted by $\{\tilde{W}_n^{j-1}\}$, a subgroup $G^j$ of $G$, $\psi^j \in H_{G^j}^1$, sequences $\{t_n^j\}_{n\in \N}\subset \R$, $\{x_n^j\}_{n\in \N} \subset \R^d$, and $\{W_n^j\}_{n\in \N} \subset H_{G^j}^1$ such that 
\begin{equation*}
\tilde{W}_n^{j-1} = \tilde{\psi}_n^j  +\tilde{W}_n^j,
\end{equation*}
where $\tilde{\psi}_n^j:= e^{it_n^j \Delta} \sum_{\cG \in G}  \cG (\tau_{x_n^j} \psi^j )/\#G $ and $\tilde{W}_n^j :=\sum_{\cG \in G}  \cG  W_n^j/\#G$ and the following statements hold. 
\begin{align*}
&e^{-it_n^j\Delta} \tau_{-\cG x_n^j} \tilde{W}_n^{j-1} \wto \frac{\cG \psi^j}{\#G/\#G^j}, \text{ and } e^{-it_n^j\Delta} \tau_{-\cG x_n^j} \tilde{W}_n^j \wto 0 \text{ in } H^1(\R^d) \text{ for all } \cG \in G.
\\
&t_n^j=0 \text{ or } t_n^j \to \pm \infty \text{ as } n \to \infty.
\\
&\cG^j x_n^j=x_n^j \text{ for all } \cG^j \in G^j \text{ and } |x_n^j - \cG x_n^j| \to \infty \text{ for all }\cG \in G\setminus G^j.
\\
&\norm[\tilde{W}_n^{j-1}]_{\dot{H}^\lambda}^2 -  \norm[ \tilde{\psi}_n^j]_{\dot{H}^\lambda}^2 -  \norm[\tilde{W}_n^j]_{\dot{H}^\lambda}^2 \to 0 \text{ as } n \to \infty,
\text{ for all }0 \leq \lambda \leq 1.
\\ 
&\norm[\tilde{W}_n^{j-1}]_{L^{p+1}}^{p+1} -  \norm[  \tilde{\psi}_n^j]_{L^{p+1}}^{p+1} -  \norm[\tilde{W}_n^j ]_{L^{p+1}}^{p+1} \to 0 \text{ as } n \to \infty.
\\
&\norm[\psi^j]_{H^1} \geq \nu A_j^{\frac{d-2\Lambda^2}{2\Lambda(1-\Lambda)}} \l( \limsup_{n \to \infty} \norm[\tilde{W}_n^{j-1}]_{H^1}\r)^{-\frac{d-2\Lambda}{2\Lambda(1-\Lambda)}},
\end{align*}
where $\nu>0$ is the constant appearing in \lemref[IS]. We call these properties $j$th properties. Here, we regard $\tilde{W}_n^0$ as $\varphi_n$. Combining $1$st, $2$nd, $\cdots$, and $J$th properties, we obtain the statements (1), (2), and (5). Note that the orthogonalities of the functionals $S_{\omega}$ and $K$ follows from the orthogonalities in norms. We prove (4). By the orthogonality of $H^1$-norm, we have
$\limsup_{n \to \infty} \norm[\tilde{W}_n^j]_{H^1} \leq  \limsup_{n \to \infty} \norm[\varphi_n]_{H^1} =a$ for all $j$.
Thus we see that
\begin{equation}
\label{3.25}
\norm[\psi^j]_{H^1} \geq \nu A_j^{\frac{d-2\Lambda^2}{2\Lambda(1-\Lambda)}} a^{-\frac{d-2\Lambda}{2\Lambda(1-\Lambda)}}.
\end{equation} 
By the orthogonality of $H^1$-norm and \lemref[lemA.3], we also have $\sum_{j=1}^{J}\norm[\psi^j]_{H^1} \leq a$. Taking the limit $J \to \infty$, we obtain $\sum_{j=1}^{\infty}\norm[\psi^j]_{H^1} \leq a$. Combining this with \eqref[3.25], we obtain
\[ \nu a^{-\frac{d-2\Lambda}{2\Lambda(1-\Lambda)}} \sum_{j=1}^{\infty} A_j^{\frac{d-2\Lambda^2}{2\Lambda(1-\Lambda)}} \leq \sum_{j=1}^{\infty}\norm[\psi^j]_{H^1} \leq a.\]
This gives us $\sum_{j=1}^{\infty} A_j^{\frac{d-2\Lambda^2}{2\Lambda(1-\Lambda)}}  \cleq a^{1+\frac{d-2\Lambda}{2\Lambda(1-\Lambda)}} < \infty$ so that $A_J\to 0$ as $J \to \infty$. 
By the H\"{o}lder inequality and the Strichartz estimate \eqref[3.1], we obtain
\begin{align*}
\limsup_{n \to \infty} \norm[e^{it\Delta} \tilde{W}_n^J]_{L^\alpha(\R:L^r)}
& \leq \limsup_{n \to \infty} \l( \norm[e^{it\Delta} \tilde{W}_n^J]_{L^q(\R:L^r)}^{\theta} \norm[e^{it\Delta} \tilde{W}_n^J]_{L^\infty(\R:L^r)}^{1-\theta} \r)
\\
& \cleq \limsup_{n \to \infty} \norm[\tilde{W}_n^J]_{H^1}^{\theta} \limsup_{n \to \infty} \norm[e^{it\Delta} \tilde{W}_n^J]_{L^\infty(\R:L^r)}^{1-\theta}
\\
&  \cleq a^{\theta} \limsup_{n \to \infty} \norm[e^{it\Delta} \tilde{W}_n^J]_{L^\infty(\R:L^r)}^{1-\theta}
\to 0
\end{align*}
as $J \to \infty$ where $\theta=2\{4-(d-2)(p-1)\}/\{d(p-1)^2\} \in (0,1)$. Thus, we obtain (4).
At last, we prove (3). By the statement (6) in \lemref[IS], $\psi^J=0$ for some $J$ implies that $\psi^j=0$ for $j\geq J$. Therefore, there exists $J \in \N \cup\{\infty\}$ such that $\psi^j \neq0$ for $j\leq J$ and $\psi^j=0$ for $j> J$. In the case of $J=1$, there is nothing to prove. We suppose $J \geq 2$. By $1$st properties, we have 
\[e^{-it_n^1\Delta} \tau_{- \cG x_n^1} \tilde{W}_n^1 \wto 0 \text{ in } H^1(\R^d) \text{ for all } \cG \in G \]
and 
\[e^{-it_n^2\Delta} \tau_{-\cG' x_n^2} \tilde{W}_n^1 \wto \frac{\cG' \psi^2}{\#G/\#G^2} \neq 0 \text{ in } H^1(\R^d) \text{ for all } \cG' \in G. \]
By \lemref[lemA.2] as $f_n=e^{-it_n^1\Delta} \tau_{- \cG x_n^1} \tilde{W}_n^1$, we get $|t_n^1-t_n^2|+|\cG x_n^1- \cG' x_n^2| \to \infty$ for all $\cG, \cG' \in G$. 
We suppose $J \geq 3$ and that the statement (3) is true for $1\leq j \neq h \leq \mu-1$ for some $\mu< J$. For $j \in \{1,2,\cdots,\mu-1 \}$, we have
\begin{align*}
\tilde{W}_n^{\mu-1}-\tilde{W}_n^{j-1}
&= \tilde{W}_n^{\mu-2} - \tilde{\psi}_n^{\mu-1} -\tilde{W}_n^{j-1}
 = \tilde{W}_n^{\mu-3} - \tilde{\psi}_n^{\mu-2} - \tilde{\psi}_n^{\mu-1} -\tilde{W}_n^{j-1}
\\
&= \cdots
=-\sum_{\nu=j}^{\mu-1}\tilde{\psi}_n^{\nu}.
\end{align*}
Therefore, we see that
\begin{align}  \label{3.26.0}
e^{-it_n^{j}\Delta} \tau_{-\cG x_n^j} \l( \tilde{W}_n^{\mu-1}-\tilde{W}_n^{j-1}\r)
&=- e^{-it_n^{j}\Delta} \tau_{-\cG x_n^j} \tilde{\psi}_n^{j} - \sum_{\nu=j+1}^{\mu-1}  e^{-it_n^{j}\Delta} \tau_{- \cG x_n^j} \tilde{\psi}_n^{\nu}.
\end{align}
Let $\{\cG_{k}^{(j)}\}_{k=1}^{\#G/\#G^j}$ be the set of left coset representatives and $\cG=\cG_{l}^{(j)} \cG^j$ for some  $l\in \{1,2,\cdots,\#G/\#G^j\}$ and $\cG^j \in G^j$.
For the first term of the right hand side in \eqref[3.26.0], we have 
\begin{align*}
e^{-it_n^{j}\Delta} \tau_{-\cG x_n^j} \tilde{\psi}_n^{j} 
&= \tau_{-\cG x_n^j} \sum_{k=1}^{\#G/\#G^j} \frac{\cG_k^{(j)} (\tau_{x_n^j} \psi^j)}{\#G/\#G^j}
\\
&=  \frac{\cG \psi^j}{\#G/\#G^j}+ \sum_{\substack{k=1\\ k\neq l}}^{\#G/\#G^j} \frac{ \tau_{-\cG_{l}^{(j)} x_n^j + \cG_{k}^{(j)}  x_n^j} \cG_k^{(j)}  \psi^j}{\#G/\#G^j} \wto \frac{\cG \psi^j}{\#G/\#G^j}.
\end{align*}
The second term of the right hand side in \eqref[3.26.0] weakly converges to $0$ in $H^1(\R^d)$ since we suppose that the statement (3) holds for  $1\leq j \neq h \leq \mu-1$. Since $ e^{-it_n^{j}\Delta} \tau_{-\cG x_n^j} \tilde{W}_n^{j-1} \wto  \cG \psi^j/(\#G/\#G^j)$, we have
\[e^{-it_n^j \Delta} \tau_{- \cG x_n^j} \tilde{W}_n^{\mu-1} \wto 0 \text{ in } H^1(\R^d) \text{ for all } \cG \in G, \]
and 
\[e^{-it_n^{\mu} \Delta} \tau_{-\cG' x_n^{\mu}} \tilde{W}_n^{\mu-1} \wto \frac{\cG' \psi^2}{\#G/\#G^2} \neq 0 \text{ in } H^1(\R^d) \text{ for all } \cG' \in G. \]
By \lemref[lemA.2] as $f_n=e^{-it_n^j \Delta} \tau_{- \cG x_n^j} \tilde{W}_n^{\mu-1}$, we get $|t_n^j-t_n^\mu|+|\cG x_n^j- \cG' x_n^\mu| \to \infty$ for all $\cG, \cG' \in G$. Therefore, we obtain the statement (3). This completes the proof. 
\end{proof}

\begin{lemma} 
\label{lem3.8}
Let $k$ be a nonnegative integer and $\varphi_{j} \in H_{G}^1(\R)$ for $j\in \{1,2,\cdots,k\}$ satisfy
\begin{align*} 
\begin{array}{ll} 
S_{\omega}(\sum_{j=1}^{k} \varphi_{j})\leq l_{\omega}^G-\delta, 
&
S_{\omega}(\sum_{j=1}^{k} \varphi_{j})\geq \sum_{j=1}^{k}  S_{\omega}(\varphi_{j}) - \eps, 
\\
K(\sum_{j=1}^{k} \varphi_{j})\geq -\eps, 
&
K(\sum_{j=1}^{k} \varphi_{j})\leq \sum_{j=1}^{k} K(\varphi_{j}) + \eps,
\end{array}
\end{align*}
for $\delta$, $\eps$ satisfying $2\eps<\delta$. Then we have $0\leq S_{\omega}(\varphi_{j}) < l_{\omega}^{G}$ and $K(\varphi_{j})\geq 0$ for all $j\in \{1,2,\cdots,k\}$. Namely, we see that $\varphi_j \in \scK_{G,\omega}^{+}$ for all $j\in \{1,2,\cdots,k\}$.
\end{lemma}

\begin{proof}
Let $J_{\omega}:=S_{\omega}-dK/4$. First, we prove the following equality.
\[ l_{\omega}^{G} = \inf \l\{ J_{\omega}(\varphi): \varphi \in H_{G}^1 \setminus \{0\}, K(\varphi) \leq 0 \r\}. \]
Let $l'$ denote the right hand side. 
By the definition of $l_{\omega}^{G}$ and $J_{\omega}$, we have
\[ l_{\omega}^{G} = \inf \l\{ J_{\omega}(\varphi): \varphi \in H_{G}^1 \setminus \{0\}, K(\varphi)=0 \r\}.\]
Therefore, we have $l_{\omega}^{G} \geq l'$. We prove $l_{\omega}^{G} \leq l'$. Take $\varphi \in H_{G}^1$ such that $K(\varphi)\leq 0$. Then, there exists $\lambda_0\leq 0$ such that $K(\varphi^{\lambda_0})=0$. Thus, we see that
\[ l_{\omega}^{G} \leq S_{\omega}(\varphi^{\lambda_0}) = J_{\omega}(\varphi^{\lambda_0}) \leq J_{\omega}(\varphi). \]
Taking the infimum for $\varphi \in H_{G}^1$ such that $K(\varphi) \leq 0$, we obtain $l_{\omega}^{G} \leq l'$. Thus, we have $l_{\omega}^{G}=l'$. Next, we prove the statement of the present lemma. We assume that there exists an $j \in \{ 1,2,\cdots, k\}$ such that $K(\varphi_{j})<0$. By $l_{\omega}^{G}=l'$ and the positivity of $J_{\omega}$, we obtain
\begin{align*}
l_{\omega}^G
& \leq \sum_{j=1}^{k} J_{\omega}(\varphi_{j})
\\
& = \sum_{j=1}^{k} \l(S_{\omega}(\varphi_{j}) -\frac{d}{4}K(\varphi_{j})\r)
\\
& = \sum_{j=1}^{k} S_{\omega}(\varphi_{j}) -   \sum_{j=1}^{k } \frac{d}{4} K(\varphi_{j})
\\
& \leq S_{\omega}\l(\sum_{j=1}^{k} \varphi_{j}\r) +\eps - \frac{d}{4} \l(K\l(\sum_{j=1}^{k} \varphi_{j}\r) - \eps \r)
\\
& \leq l_{\omega}^G-\delta +\eps  +\eps 
\\
& < l_{\omega}^G.
\end{align*}
This is a contradiction. So, $K(\varphi_{j})\geq 0$ for all $j\in \{1,2,\cdots,k\}$. Moreover, for any $j\in \{1,2,\cdots,k\}$, we have
\begin{align*}
S_{\omega}(\varphi_{j}) = J_{\omega}(\varphi_{j}) + \frac{d}{4} K(\varphi_{j}) \geq 0,
\end{align*}
and 
\begin{align*}
S_{\omega}(\varphi_{j}) \leq \sum_{j=1}^{k}  S_{\omega}( \varphi_{j})  \leq S_{\omega}\l( \sum_{j=1}^{k} \varphi_{j}\r) + \eps \leq l_{\omega}^G -\delta +\eps <l_{\omega}^G.
\end{align*}
This completes the proof.
\end{proof}

We collect lemmas for nonlinear profiles.

\begin{lemma} \label{lem3.9}
Let $\{x_n\}$ be a sequence, $\psi \in H^1$, and $U$ be a solution of \eqref[NLS] with the initial data $\psi$. Then, we have
\[ U_n (t)= e^{it\Delta} \tau_{x_n} \psi +i \int_{0}^{t} e^{i(t-s)\Delta} (|U_n(s)|^{p-1} U_n(s)) ds, \]
where $U_n (t,x) := U(t,x-x_n)$.
\end{lemma}

\lemref[lem3.9] follows from the space translation invariance of the equation \eqref[NLS].

\begin{lemma} \label{lem3.10.1}
Let $\{t_n\} $ satisfy $t_n \to \pm \infty$, $\{x_n\}$ be a sequence, $\psi \in H^1$, and $U$ be a solution of \eqref[NLS] satisfying
\[\norm[U_{\pm}(t) - e^{it\Delta} \psi]_{H^1} \to 0 \text{ as } t \to \pm \infty \]
Then, we have
\[ U_{\pm,n} (t)= e^{it\Delta} e^{it_n \Delta} \tau_{x_n} \psi +i \int_{0}^{t} e^{i(t-s)\Delta} (|U_{\pm,n}(s)|^{p-1} U_{\pm,n}(s)) ds + e_{\pm,n}(t), \]
where $U_{\pm,n} (t,x) := U_{\pm}(t+t_n,x-x_n)$ and $\norm[e_{\pm,n}]_{L^{\alpha}(\R:L^r)} \to 0 $ as $n \to \infty$. 
\end{lemma}

\begin{proof}
Since $U_{\pm,n}$ is a solution of \eqref[NLS] with the initial data $\tau_{x_n}U_{\pm}(t_n)$ by the time and space translation invariance, we have
\begin{align*} 
e_{\pm,n}(t)
&=U_{\pm,n} (t)- e^{it\Delta} e^{it_n \Delta} \tau_{x_n} \psi -i \int_{0}^{t} e^{i(t-s)\Delta} (|U_{\pm,n}(s)|^{p-1} U_{\pm,n}(s)) ds 
\\
&= e^{it\Delta}\tau_{x_n} U_{\pm}(t_n) 
- e^{it\Delta} e^{it_n \Delta} \tau_{x_n} \psi. 
\end{align*}
By the Strichartz estimate, 
\[ \norm[e_{\pm,n}]_{L^{\alpha}(\R:L^r)}  \cleq \norm[U_{\pm}(t_n) 
- e^{it_n \Delta}  \psi]_{H^1} \to 0 \text{ as } n \to \infty.  \]
This completes the proof. 
\end{proof}


\subsection{Construction of a critical element and Rigidity}

By the definition of $\bS_{\omega}^{G}$, we have $\bS_{\omega}^{G} \leq l_{\omega}^{G}$.
\lemref[lem2.2] and \lemref[SDS] give $\bS_{\omega}^{G}>0$.  We prove $\bS_{\omega}^{G}=\min\{m_{\omega}^{G}, l_{\omega}^{G}\} $ by contradiction argument so that we suppose $\bS_{\omega}^{G} < \min\{m_{\omega}^{G}, l_{\omega}^{G}\} $. 

\begin{proposition} \label{prop3.11}
Assume $\bS_{\omega}^{G} <  \min\{m_{\omega}, l_{\omega}^{G}\} $. Then, there exists a global solution $u^c$  to \eqref[NLS] with $G$-invariance such that $S_{\omega}(u^c)=\bS_{\omega}^{G}$ and $\norm[u^c]_{L^\alpha(\R: L^r)}=\infty$. 
\end{proposition}

We call $u^c$ a critical element. 

\begin{proof}
By the definition of $\bS_{\omega}^{G}$ and the assumption of $\bS_{\omega}^{G} <  \min\{m_{\omega}, l_{\omega}^{G}\} $, there exists a sequence $\{\varphi_n\} \in \scK_{G,\omega}^{+}$ satisfying $\bS_{\omega}^{G} < S_{\omega}(\varphi_n) <\min\{m_{\omega}, l_{\omega}^{G}\}$, $S_{\omega}(\varphi_n) \searrow \bS_{\omega}^{G}$, and $u_n \not\in L^\alpha(\R: L^r(\R^d))$, where $u_n$ is a global solution with the initial data $\varphi_n$. Since $\{\varphi_n\}$ is bounded in $H^1(\R^d)$, we apply the linear profile decomposition with $G$-invariance, \propref[LPD], to the sequence $\{\varphi_n\}$ and then we obtain
\[ \varphi_n = \sum_{j=1}^J \tilde{\psi}_n^j  +  \tilde{W}_n^J, \]
where we recall that $\tilde{\psi}_n^j = \sum_{\cG \in G} e^{i t_n^j \Delta}\cG (\tau_{x_n^j} \psi^j )/\#G$ and $\tilde{W}_n^J = \sum_{\cG \in G} \cG  W_n^J/\#G$.
We also see that
\begin{align*}
S_{\omega}(\varphi_n) 
&=\sum_{j=1}^J S_{\omega}\l(  \tilde{\psi}_n^j  \r)+ S_{\omega}\l( \tilde{W}_n^J\r)+o(1),
\\
K(\varphi_n) 
&=\sum_{j=1}^J K\l(  \tilde{\psi}_n^j  \r)+ K\l( \tilde{W}_n^J\r)+o(1),
\end{align*} 
where $o(1) \to 0$ as $n \to \infty$. By these decompositions, we have
\begin{align*}
&S_{\omega}\l(\varphi_n\r)\leq l_{\omega}^{G}-\delta, 
\\
&S_{\omega}\l(\varphi_n\r)\geq \sum_{j=1}^{J}  S_{\omega}\l(\tilde{\psi}_n^j  \r) + S_{\omega}\l( \tilde{W}_n^J\r) - \eps, 
\\
&K\l(\varphi_n\r)\geq 0 > -\eps, 
\\
&K\l(\varphi_n\r)\leq \sum_{j=1}^{J} K\l(\tilde{\psi}_n^j  \r) +K\l( \tilde{W}_n^J\r)+ \eps,
\end{align*}
for large $n$ where $\delta=l_{\omega}^{G} - S_{\omega}(\varphi_1)$ and $\eps>0$ satisfies $2 \eps < \delta$. Therefore, \lemref[lem3.8] gives us that 
$ \tilde{\psi}_n^j \in \scK_{G,\omega}^{+}$ for all $j\in\{1,2,\cdots,J\}$ and $\tilde{W}_n^J  \in \scK_{G,\omega}^{+}$. Thus, for any $J$, we obtain
\[ \bS_{\omega}^{G} = \lim_{n \to \infty} S_{\omega}(\varphi_n) \geq \sum_{j=1}^{J} \limsup_{n \to \infty} S_{\omega}\l(   \tilde{\psi}_n^j  \r). \]
We prove $\bS_{\omega}^{G}= \limsup_{n \to \infty} S_{\omega}( \tilde{\psi}_n^j )$ for some $j$ by a contradiction argument. 
We assume that $\bS_{\omega}^{G}= \limsup_{n \to \infty} S_{\omega}( \tilde{\psi}_n^j )$ fails for all $j$. Namely, we assume that $ \limsup_{n \to \infty} S_{\omega}( \tilde{\psi}_n^j )<\bS_{\omega}^{G}$ for all $j$. By reordering, we can choose $0 \leq J_1 \leq J_2 \leq  J$ such that
\begin{align*}
\begin{array}{cc}
1\leq j \leq J_1: & \qquad t_n^j=0, \ \forall n 
\\
J_1+1 \leq j \leq J_2: & \qquad \lim_{n\to\infty }t_n^j=-\infty,  
\\
J_2 +1\leq j \leq J: & \qquad \lim_{n\to\infty }t_n^j=+\infty.
\end{array}
\end{align*} 
Above we are assuming that if $a>b$ then there is no $j$ such that $a \leq j \leq b$. Note that $J_1\in \{0,1\}$ by the orthogonality of the parameter $\{t_n^j\}$ (see \propref[LPD] (3)). We only consider the case $J_1=1$ since the case $J_1=0$ is easier. By the assumption of the contradiction argument and $t_n^1=0$, we have $0< \limsup_{n \to \infty} S_{\omega}( \sum_{\cG \in G}  \cG (\tau_{x_n^1} \psi^1 )/\#G)<\bS_{\omega}^{G}$. By the choice of $\{x_n^1\}$ and \lemref[lemA.3], 
\[  \limsup_{n \to \infty} S_{\omega}\l( \sum_{\cG \in G} \frac{ \cG (\tau_{x_n^1} \psi^1 )}{\#G}\r)=\frac{\#G}{\#G^1}  S_{\omega}\l( \frac{ \psi^1 }{\#G/\#G^1}\r) < \bS_{\omega}^{G} < m_{\omega}^{G} . \]
Therefore, $S_{\omega}\l( \psi^1 /(\#G/\#G^1)\r)< \bS_{\omega}^{G^1}.$ By the definition of $\bS_{\omega}^{G^1}$, the solution $U^1$ to \eqref[NLS] with the initial data $\psi^1/(\#G/\#G^1)$ belongs to $L^{\alpha}(\R:L^r(\R^d))$. 

For $j \in [J_1+1,J_2]$, we have
\begin{align*}  
m_{\omega}^{G}
&>\bS_{\omega}^{G} 
> \limsup_{n \to \infty} S_{\omega}\l(\tilde{\psi}_n^j \r)
\\
&= \limsup_{n \to \infty} \l( \frac{1}{2}\norm[ \sum_{\cG \in G} \frac{ \cG (\tau_{x_n^j} \psi^j )}{\#G}]_{\dot{H}^1}^2 
+\frac{\omega}{2} \norm[\sum_{\cG \in G} \frac{ \cG (\tau_{x_n^j} \psi^j )}{\#G}]_{L^2}^2\r)
\\
&\qquad - \frac{1}{p+1} \liminf_{n \to \infty} \norm[e^{i t_n^j \Delta}  \sum_{\cG \in G} \frac{ \cG (\tau_{x_n^j} \psi^j )}{\#G}]_{L^{p+1}}^{p+1}
\\
&=\frac{\#G}{\#G^j} \l( \frac{1}{2}\norm[  \frac{ \psi^j }{\#G/\#G^j}]_{\dot{H}^1}^2 
+\frac{\omega}{2} \norm[ \frac{ \psi^j }{\#G/\#G^j}]_{L^2}^2 \r), 
\end{align*}
where we use $\norm[e^{it_n \Delta} \phi]_{L^{p+1}} \to 0$ as $n \to \infty$ (see \cite[Corollary 2.3.7]{Caz03}) and \lemref[lemA.3].  
This inequality implies that $\psi^j/(\#G/\#G^j)$ satisfies the assumption of \lemref[lem3.4] as $G=G^j$, where we note that $\bS_{\omega}^{G^j} \leq l_{\omega}^{G^j}$. Thus, we obtain the global solution $U_{-}^j$ to \eqref[NLS] such that $U_{-}^{j}(0) \in \scK_{G,\omega}^{+}$ and
\[ \norm[ U_{-}^j(t) - e^{it\Delta} \frac{ \psi^j }{\#G/\#G^j}]_{H^1} \to 0 \text{ as } t \to - \infty. \]
Moreover, $U_{-}^j$  belongs to $L^\alpha(\R:L^r(\R^d))$ by the definition of $\bS_{\omega}^{G^j}$ since we have
\[S_{\omega}(U_{-}^j)=\frac{1}{2}\norm[  \frac{ \psi^j }{\#G/\#G^j}]_{\dot{H}^1}^2 
+\frac{\omega}{2} \norm[ \frac{ \psi^j }{\#G/\#G^j}]_{L^2}^2< \bS_{\omega}^{G^j}.\]
For $j \in [J_2+1,J]$, by the similar argument, we obtain a global solution $U_{+}^j$ such that
$U_{+}^{j}(0) \in \scK_{G,\omega}^{+}$, $U_{+}^j \in L^\alpha(\R:L^r(\R^d))$, and
\[ \norm[ U_{+}^j(t) - e^{it\Delta} \frac{ \psi^j }{\#G/\#G^j}]_{H^1} \to 0 \text{ as } t \to \infty. \]
We define 
\begin{equation*}
U^j := \l\{
\begin{array}{ll}
U^1, & \text{if } j=1,
\\
U_{-}^j, & \text{if } j \in [2,J_2],
\\
U_{+}^j, & \text{if } j \in [J_2 + 1,J],
\end{array}
\r.
\text{ and } U_n^j (t,x):= U^j(t+t_n^j, x-x_n^j).
\end{equation*}
Moreover, we define 
\begin{align*}
u_n^J :=\sum_{j=1}^{J} \sum_{k=1}^{\#G/\#G^j} \cG_k^{(j)} U_n^j,
\end{align*}
where $\{\cG_{k}^{(j)}\}_{k=1}^{\#G/\#G^j}$ be the set of left coset representatives.
Then $u_n^J$ satisfies 
\begin{align*}
&i \del_t u_n^J + \Delta u_n^J + \l|u_n^J\r|^{p-1}u_n^J = e_n^J,
\\
&e_n^J = \l|u_n^J\r|^{p-1}u_n^J -\sum_{j=1}^{J} \sum_{k=1}^{\#G/\#G^j}  \l|\cG_k^{(j)} U_n^j\r|^{p-1} \cG_k^{(j)} U_n^j.
\end{align*}
Moreover, we have 
\[u_n(0)-u_n^J(0) = \sum_{j=1}^{J} \sum_{k=1}^{\#G/\#G^j} \cG_k^{(j)} \l( e^{it_n^j \Delta} \tau_{x_n^j} \frac{\psi^j}{\#G/\#G^j} -\tau_{x_n^j} U^j(t_n^j) \r)+ \tilde{W}_n^J. \]
To apply the perturbation lemma, \lemref[Perturb], we prove the following inequalities hold for large $n$.
\begin{align}
\label{3.26}
&\norm[u_n^J]_{L^{\alpha}(\R;L^r)} \leq A,
\\
\label{3.27}
& \norm[e_n^J ]_{L^{\beta'}(R:L^{r'})} \leq \eps(A),
\\
\label{3.28}
&\norm[e^{it\Delta}( u_n(0)-u_n^J(0))]_{L^\alpha(\R:L^r)} \leq \eps(A).
\end{align}
We prove \eqref[3.26]. By the definition of $U_n^j$, we have
\begin{align*}
u_n^J (t)
&= \sum_{j=1}^{J} \sum_{k=1}^{\#G/\#G^j} \cG_k^{(j)} U_n^j(t)
\\
&= \sum_{k=1}^{\#G/\#G^1} \cG_k^{(1)} (\tau_{x_n^1} U^1(t)) + \sum_{j=2}^{J_2} \sum_{k=1}^{\#G/\#G^j} \cG_k^{(j)} (\tau_{x_n^j} U_{-}^j(t+t_n^j)) + \sum_{j=J_2+1}^{J} \sum_{k=1}^{\#G/\#G^j} \cG_k^{(j)}  (\tau_{x_n^j} U_{+}^j(t+t_n^j)).
\end{align*}
By \lemref[lemA.4], we obtain
\[ \limsup_{n \to \infty} \norm[u_n^J]_{L^\alpha(\R:L^r)} \leq 2  \sum_{j=1}^{J} \sum_{k=1}^{\#G/\#G^j} \norm[\cG_k^{(j)}  U^j]_{L^\alpha(\R:L^r)} = 2  \sum_{j=1}^{J} \frac{\#G}{\#G^j} \norm[U^j]_{L^\alpha(\R:L^r)}, \]
where we use $|t_n^j-t_n^h|\to \infty$ or $|\cG x_n^j - \cG' x_n^h| \to \infty $ for all $\cG,\cG'\in G$ if $j\neq h$ and also use $|\cG_k^{(j)} x_n^j -\cG_l^{(j)}  x_n^j| \to \infty$ if $k\neq l$. 
By (5) in \propref[LPD] and \lemref[lemA.3], we have
\[ \norm[\varphi_n]_{H^1}^2=\sum_{j=1}^{J} \frac{\#G}{\#G^j} \norm[\frac{\psi^j}{\#G/\#G^j}]_{H^1}^2+\norm[\tilde{W}_n^J]_{H^1}^2 + o_n(1).\]
Therefore, $\sup_{n \in \N} \norm[\varphi_n]_{H^1}^2 < \infty$ implies that there exists a finite set $\scJ$ such that $\norm[\psi^j/(\#G/\#G^j)]_{H^1} < \eps_{sd}$ for $j \not\in \scJ$, where $\eps_{sd}$ is a constant appearing in \propref[SDS]. Thus, we get 
\begin{align*}
\limsup_{n \to \infty} \norm[u_n^J]_{L^\alpha(\R:L^r)} 
&\cleq  \sum_{j=1}^{J}\norm[U^j]_{L^\alpha(\R:L^r)}
\\
&=  \sum_{j\in \scJ}\norm[U^j]_{L^\alpha(\R:L^r)}+ \sum_{j \not\in \scJ}\norm[U^j]_{L^\alpha(\R:L^r)}
\\
& \cleq  \sum_{j\in \scJ}\norm[U^j]_{L^\alpha(\R:L^r)}+\sum_{j \not\in \scJ} \norm[\frac{\psi^j}{\#G/\#G^j}]_{H^1}
\\
& \cleq  \sum_{j\in \scJ}\norm[U^j]_{L^\alpha(\R:L^r)}+ \limsup_{n \to \infty} \norm[\varphi_n]_{H^1}
\\
& \leq  A < \infty.
\end{align*}
We prove \eqref[3.28]. By the triangle inequality, the Strichartz estimate, the definition of $U^j$, (4) in \propref[LPD], and Lemmas \ref{lem3.9} and \ref{lem3.10.1}, we have
\begin{align*} 
&\norm[e^{it\Delta}( u_n(0)-u_n^J(0))]_{L^\alpha(\R:L^r)}
\\
&\leq \sum_{j=1}^{J} \sum_{k=1}^{\#G/\#G^j} \norm[e^{it\Delta}\l( \tau_{x_n^j} U^j(t_n^j) -  e^{it_n^j \Delta} \tau_{x_n^j} \frac{\psi^j}{\#G/\#G^j}\r) ]_{L^\alpha(\R:L^r)}
 +\norm[e^{it\Delta}\tilde{W}_n^J]_{L^\alpha(\R:L^r)}
\\
&\leq \sum_{j=1}^{J} \sum_{k=1}^{\#G/\#G^j} \norm[\tau_{x_n^j} U^j(t_n^j) -  e^{it_n^j \Delta} \tau_{x_n^j} \frac{\psi^j}{\#G/\#G^j} ]_{H^1}
 +\norm[e^{it\Delta}\tilde{W}_n^J]_{L^\alpha(\R:L^r)}
\\
& \leq \eps \leq \eps(A),
\end{align*}
for large $n$ and $J$. 
We prove \eqref[3.27]. In general, the following inequality holds. 
\[ \l| \l| \sum_{j=1}^J z^j \r|^{p-1} \sum_{j=1}^J z^j -\sum_{j=1}^J \l| z^j \r|^{p-1} z^j \r| \leq C_J \sum_{1\leq j\neq h \leq J} |z^j|^{p-1} |z^h|. \]
This implies that 
\[ \norm[e_n^J ]_{L^{\beta'}(\R:L^{r'})} \leq C_J \sum_{1\leq j\neq h \leq J} \norm[|U_n^j|^{p-1} |U_n^h|]_{L^{\beta'}(\R:L^{r'})}. \]
An approximation argument and $|t_n^j-t_n^h|\to \infty$ or $|\cG x_n^j - \cG' x_n^h| \to \infty $ for all $\cG,\cG'\in G$ if $j\neq h$ and also use $|\cG_k^{(j)} x_n^j -\cG_l^{(j)} x_n^j| \to \infty$ if $k\neq l$ give us $\norm[|U_n^j|^{p-1} |U_n^h|]_{L^{\beta'}(\R:L^{r'})} \to 0$ as $n \to \infty$. Thus, we obtain \eqref[3.27]. Applying \lemref[Perturb], we conclude that $u_n$ scatters. However, this contradicts the definition of $\{\varphi_n\}$. Therefore, there exists $j $ such that $\bS_{\omega}^{G}= \limsup_{n \to \infty} S_{\omega}(\tilde{\psi}_n^j)$. We may assume $j=1$. The linear profile decomposition as $J=1$ and $\tilde{W}_n^1 \in \scK_{G,\omega}^{+}$ imply $\norm[\tilde{W}_n^1]_{L^\infty(\R:H^1)} \to 0$ as $n \to \infty$ by \lemref[lem2.2]. 
Therefore, we see that
\begin{align*}
&\varphi_n = \tilde{\psi}_n^1 + \tilde{W}_n^1, 
\\
& \norm[\tilde{W}_n^1]_{L^\infty(\R:H^1)} \to 0,
\\
& \bS_{\omega}^{G} = \lim_{n \to \infty} S_{\omega}(\tilde{\psi}_n^1).
\end{align*}
We assume that there exists $G^1\subsetneq G$ such that $x_n^1 =\cG^1 x_n^1$ for all $\cG^1 \in G^1$ and $|x_n^1 - \cG x_n^1| \to \infty$ for all $\cG \in G\setminus G^1$. Let $U$ be a global solution of \eqref[NLS] with the initial data $\psi^1/(\#G/\#G^1)$ if $t_n^1=0$ or the final data $\psi^1/(\#G/\#G^1)$ if $|t_n^1| \to \infty$. Then,  by the definition of $\bS_{\omega}^{G^1}$, $U$ belongs to $L^{\alpha}(\R:L^r(\R^d))$ since we have, by \lemref[lemA.3], 
\[ \lim_{n \to \infty} S_{\omega}(\tilde{\psi}_n^1) =\lim_{n \to \infty} \frac{\#G}{\#G^1} S_{\omega}\l( e^{it_n^1 \Delta} \frac{\psi^1}{\#G/\#G^1}\r)= \bS_{\omega}^{G} <m_{\omega}^{G} \leq \frac{\#G}{\#G^1}  \bS_{\omega}^{G^1} . \]
By \lemref[Perturb] again, this contradicts that $u_n$ does not belong to $L^{\alpha}(\R:L^r(\R^d))$. Thus, $G^1=G$. This means that $\psi^1$ and $W_n^1$ are $G$-invariant, $x_n^1 = \cG x_n^1$ for all $\cG \in G$, and we see that
\[ \varphi_n = e^{it_n^1 \Delta} \tau_{x_n^1} \psi^1 + W_n^1. \]
Let $u^c$ be a global solution of \eqref[NLS] with the initial data $\psi^1$ if $t_n^1=0$ or the final data $\psi^1$ if $|t_n^1| \to \infty$. Then, $u^c$ is $G$-invariant. We prove $\norm[u^c]_{L^\alpha(\R:L^r)} = \infty$. Suppose that $\norm[u^c]_{L^\alpha(\R:L^r)} < \infty$.
We observe that $\varphi_n - \tau_{x_n^1}u^c(t_n^1)= e^{it_n^1 \Delta} \tau_{x_n^1}\psi^1 - \tau_{x_n^1}u^c(t_n^1)+W_n^1$, so that we have
\[ \norm[e^{it\Delta} \l( \varphi_n - \tau_{x_n^1}u^c(t_n^1)\r)]_{L^{\alpha}(\R;L^r)} \to 0 \text{ as } n \to \infty. \]
By \lemref[Perturb], we see that $u_n \in L^\alpha(\R:L^r(\R))$ for large $n$, which is absurd. Thus, we get $\norm[u^c]_{L^\alpha(\R:L^r)} = \infty$. Moreover, we have $S_{\omega}(u^c)=\lim_{n \to \infty}S_{\omega}(e^{it_n^1} \psi^1)=\bS_{\omega}^{G}$. Thus, we get a critical element $u^c$. 
\end{proof}

We say that the solution $u$ is a forward critical element if $u$ is a critical element and satisfies $\norm[u]_{L^{\alpha}([0,\infty):L^r)}=\infty$. In the same manner, we define a backward critical element. We only prove extinction of the forward critical element since that of the backward critical element can be obtained by the similar argument based on time reversibility. The extinction contradicts \propref[prop3.11].

\begin{lemma} \label{lem3.12}
Let $u$ be a forward critical element. 
There exists a continuous function $x:[0,\infty) \to \R^d$ such that $\cG x(t) = x(t)$ for all $\cG \in G$ and $\{u(t,\cdot-x(t)):t\in [0,\infty) \}$ is precompact in $H^1(\R^d)$. 
\end{lemma}

The above lemma can be obtained by the same argument as in \cite[Proposition 3.2]{DHR08} noting $u$ is $G$-invariant and $\{x_n^1\}$, which appears in the profile decomposition, satisfies $\cG x_n^1 =x_n^1$ for all $\cG \in G$. 

\begin{lemma} \label{lem3.13}
Let $u$ be a solution to \eqref[NLS] satisfying that there exists a continuous function $x:[0,\infty) \to \R^d$ such that $\{u(t,\cdot-x(t)):t\in [0,\infty) \}$ is precompact in $H^1(\R^d)$. Then, for any $\eps>0$, there exists $R=R(\eps)>0$ such that 
\begin{equation} \label{3.30}
\int_{|x+x(t)|>R}  |\nabla u(t,x)|^2  + |u(t,x)|^2 + |u(t,x)|^{p+1} dx <\eps \text{ for any } t\in [0,\infty).
\end{equation}
\end{lemma}

It can be obtained by using directly the argument of \cite[Corollary 3.3]{DHR08}. 

\begin{lemma} \label{lem3.14}
Let $u$ be a forward critical element. Then, the momentum must be $0$, \textit{i.e.} $P(u)=0$. 
\end{lemma}

\begin{proof}
First, we prove $\cG P(u)=P(u)$ for all $\cG \in G$. By the $G$-invariance of $u$, we see that
\begin{align*}
P(u) 
& = P(\cG^{-1} u)
 = \im \int_{\R^d} \overline{e^{i\theta} u( \cG x)} \nabla \{ e^{i\theta}  u(\cG x) \}dx
\\
&= \cG \im \int_{\R^d} \overline{u( \cG x)} \nabla  u(\cG x)dx
= \cG \im \int_{\R^d} \overline{u( x)} \nabla  u(x)dx = \cG P(u).
\end{align*}
Therefore, the Galilean transformation
\[ u_{\xi_0} (t,x):= e^{i(x\cdot \xi_0 - |\xi_0|^2t)} u(t,x- 2t \xi), \]
where $\xi_0 = - P(u)/M(u)$, conserves the $G$-invariance of the solution. The rest of the proof is same as in \cite[Proposition 4.1]{DHR08} and \cite[Proposition 4.1 (iii)]{AN13}.
\end{proof}

We use the following lemma to prove the rigidity lemma, \lemref[lem3.16].

\begin{lemma}
Let $u$ be a solution to \eqref[NLS] on $[0,\infty)$ such that $P(u)=0$ and there exists a continuous $x:[0,\infty) \to \R^d$ such that, for any $\eps>0$, there exists $R=R(\eps)>0$ such that 
\[  \int_{|x+x(t)|>R}  |\nabla u(t,x)|^2  + |u(t,x)|^2 + |u(t,x)|^{p+1} dx <\eps \text{ for any } t\in [0,\infty). \] Then, we have
\[ \frac{x(t)}{t} \to 0 \text{ as } t \to \infty. \]
\end{lemma}

This follows from \cite[Lemma 5.1]{DHR08}, \cite[Proof of Theorem 7.1, Step1]{FXC11}.

\begin{lemma}[Rigidity]
\label{lem3.16}
If the solution $u$ with $G$-invariance satisfies the following properties, then $u = 0$.
\begin{enumerate}
\item $u_0 \in \cK_{G,\omega}^{+}$.
\item $P(u)=0$.
\item There exists a continuous $x:[0,\infty) \to \R^d$ such that $ \cG x(t)=x(t)$ for all $t \in[0,\infty)$ and $\cG \in G$ and, for any $\eps>0$, there exists $R=R(\eps)>0$ such that 
\[  \int_{|x+x(t)|>R}  |\nabla u(t,x)|^2  + |u(t,x)|^2 + |u(t,x)|^{p+1} dx <\eps \text{ for any } t\in [0,\infty). \]
\end{enumerate}
\end{lemma}

For the proof of \lemref[lem3.16], see \cite[Theorem 6.1]{DHR08} and \cite[Theorem7.1]{FXC11}. 

Combining Lemmas \ref{lem3.12}, \ref{lem3.13}, and \ref{lem3.14}, the forward critical element satisfies the assumption (1)--(3) in \lemref[lem3.16]. 
The result by \lemref[lem3.16] contradicts $S_{\omega}(u)=\bS_{\omega}^{G}>0$.
This contradiction comes from the assumption $\bS_{\omega}^{G} <  \min\{m_{\omega}^{G}, l_{\omega}^{G}\}$ in \propref[prop3.11]. This means $\bS_{\omega}^G = \min\{m_{\omega}^{G},l_{\omega}^{G}\}$, which completes the proof of \thmref[thm1.1] (1) (i).


\section{Proof of the Scattering Result for the infinite group invariant solutions}

In this section, we prove \thmref[thm1.1] (1) (ii). Let a subgroup $G$ of $\R/2\pi\Z \times O(d)$ be infinite and satisfy that the embedding $H_{G}^1 \hookrightarrow L^{p+1}(\R^d)$ is compact throughout this section. 

First, we prove that the sequence $\{\tilde{x}_n\}$ in \eqref[3.20] is bounded by a contradiction argument. We suppose that $\{\tilde{x}_n\}$ is unbounded. We may assume that $|\tilde{x}_n| \to \infty$ as $n \to \infty$. Since $\{\varphi_n\}$  is bounded in $H^1$, we can find $\tilde{\psi} \in H_{G}^1$ and $\psi \in H^1(\R^d)$ such that 
\begin{align*}
e^{-i\tilde{t}_n\Delta} \varphi_n  &\wto \tilde{\psi} \text{ in } H^1(\R^d),
\\
e^{-i\tilde{t}_n\Delta} \tau_{-\tilde{x}_n} \varphi_n  & \wto \psi \text{ in } H^1(\R^d).
\end{align*}
Since the embedding $H_{G}^1 \hookrightarrow L^{p+1}(\R^d)$ is compact, we see that $e^{-i\tilde{t}_n\Delta} \varphi_n  \to \tilde{\psi}$ in $L^{p+1}(\R^d)$. On the other hand, for any $R>0$, we have $e^{-i\tilde{t}_n\Delta}\tau_{-\tilde{x}_n} \varphi_n  \to \psi$ in $L^{p+1}(B_R)$, where $B_R$ is a ball of radius $R$ centered at the origin. 
These limits imply that 
\begin{align*}
\norm[\tau_{-\tilde{x}_n}\tilde{\psi} -\psi  ]_{L^{p+1}(B_R)}
&\leq \norm[\tau_{-\tilde{x}_n}\tilde{\psi}- e^{-i\tilde{t}_n\Delta} \tau_{-\tilde{x}_n} \varphi_n ]_{L^{p+1}(B_R)}
+\norm[e^{-i\tilde{t}_n\Delta}\tau_{-\tilde{x}_n} \varphi_n -\psi  ]_{L^{p+1}(B_R)}
\\
&\to 0.
\end{align*}
Since $|\tilde{x}_n| \to \infty$ as $n \to \infty$, we have $\norm[\tau_{-\tilde{x}_n} \psi]_{L^{p+1}(B_R)} \to 0$ as $n \to \infty$. Therefore, we see that, for any $R>0$,
\[ \norm[\psi  ]_{L^{p+1}(B_R)} \leq \norm[\tau_{-\tilde{x}_n}\tilde{\psi} -\psi  ]_{L^{p+1}(B_R)} +\norm[\tau_{-\tilde{x}_n} \psi]_{L^{p+1}(B_R)}  \to 0. \]
This means that $\psi =0$. However, we have $\psi \neq 0$ by \eqref[3.20]. This is a contradiction.

Thus, we can take $x_n:=0$ for all $n \in \N$ in the linear profile decomposition lemma. The rest of the proof is same as in the radial case. See \cite{HR08} for details.


\appendix


\section{Lemmas}
\label{secA}

\begin{lemma}
\label{lemA.0}
Let $G$ be a subgroup of $\R/2\pi\Z \times O(d)$ and $\{\tilde{x}_n\}$ be a sequence. Then, there exists a subgroup $G'$ of $G$ such that the sequence $\{\tilde{x}_n - \cG' \tilde{x}_n\}$ is bounded for all $\cG' \in G'$ and $|\tilde{x}_n - \cG \tilde{x}_n| \to \infty$ as $n \to \infty$ for all $\cG \in G \setminus G'$.
\end{lemma}

\begin{proof}
It is easy to check that there exists a subset $G'$ of $G$ such that the sequence $\{\tilde{x}_n - \cG' \tilde{x}_n\}$ is bounded for $\cG' \in G'$ and the sequence $\{\tilde{x}_n - \cG \tilde{x}_n\}$ is unbounded for $\cG \in G \setminus G'$. We prove that $G'$ is a group. 
Let $\langle G' \rangle$ denote the group generated by $G'$. For any $\cG \in \langle G' \rangle$, there exist $k \in \N$ and $\{\cG_1', \cdots, \cG_k'\} \subset G'$ such that $\cG  = (\cG_1')^{s_1} (\cG_2')^{s_2} \cdots (\cG_k')^{s_k}$, where $s_j$ denotes either $1$ or $-1$ for $j \in \{ 1,2,\cdots, k\}$. By the triangle inequality, we have
\begin{align*} 
|\tilde{x}_n - \cG \tilde{x}_n| 
&=|\tilde{x}_n -(\cG_1')^{s_1} (\cG_2')^{s_2} \cdots (\cG_k')^{s_k} \tilde{x}_n| 
\\
& \leq |\tilde{x}_n - (\cG_1')^{s_1} \tilde{x}_n| + |(\cG_1')^{s_1}  \tilde{x}_n -(\cG_1')^{s_1} (\cG_2')^{s_2} \tilde{x}_n|  
\\
& \quad + |(\cG_1')^{s_1} (\cG_2')^{s_2}\tilde{x}_n - (\cG_1')^{s_1} (\cG_2')^{s_2}(\cG_3')^{s_3} \tilde{x}_n| 
\\
& \qquad+ \cdots +|(\cG_1')^{s_1} (\cG_2')^{s_2} \cdots (\cG_{k-1}')^{s_{k-1}}  \tilde{x}_n -(\cG_1')^{s_1} (\cG_2')^{s_2} \cdots (\cG_k')^{s_k} \tilde{x}_n| 
\\
& \leq  |\tilde{x}_n - (\cG_1')^{s_1}  \tilde{x}_n| + |\tilde{x}_n -  (\cG_2')^{s_2}  \tilde{x}_n|  
+ |\tilde{x}_n -  (\cG_3')^{s_3}  \tilde{x}_n| + \cdots +| \tilde{x}_n -  (\cG_k')^{s_k}  \tilde{x}_n|. 
\end{align*}
Therefore, $\{ \tilde{x}_n - \cG \tilde{x}_n \}$ is bounded. This means that $ \langle G' \rangle \subset G'$ and thus $G'$ is a group. 
\end{proof}

\begin{lemma}
\label{lemA}
Let $k \in \N$ and $\cA$ be a $kd\times d$-matrix. We assume that a sequence $\{\tilde{x}_n\}\subset \R^d$ satisfies that there exists $\bar{x} \in \R^{kd}$ such that $\cI \tilde{x}_n - \cA \tilde{x}_n \to \bar{x}$ where $\cI$ is a $kd\times d$-matrix such that 
\[ \l. \cI = \begin{pmatrix} \cI_d \\ \cI_d \\ \vdots \\ \cI_d \end{pmatrix} \r\}k.\]
Then, there exist $\{x_n\}\subset \R^d$ and $x_\infty \in \R^d$ such that 
\begin{equation*}
\l\{
\begin{array}{l}
\cA x_n = \cI x_n,
\\
x_n - \tilde{x}_n \to x_\infty. 
\end{array}
\r.
\end{equation*}
\end{lemma}

\begin{proof}
It is well known that there exist $kd\times kd$-matrix $\cP$ and $d\times d$-matrix $\cQ$ such that
\begin{equation}
\label{B1}
\cP(\cI-\cA)\cQ=
\underbrace{
\l(
\begin{array}{ccc}
1 &   &   
\\
   & \ddots &   
\\
   &   & 1 
\\
   & 0 &   
\end{array}
\r.
}_{r}
\underbrace{
\l.
\begin{array}{ccc}
a_{11} & \cdots & a_{1d-r}
\\
\vdots &   &   \vdots
\\
a_{r1} & \cdots & a_{rd-r}
\\
  & 0   &     
\end{array}
\r)}_{d-r}
\begin{array}{l}
\Bigg\} r
\\
\} kd-r
\end{array},
\end{equation}
where $r=\rank (\cI-\cA)$. We set $\tilde{y}_n := \cQ^{-1} \tilde{x}_n$, $\cB:=\cP(\cI-\cA)\cQ$, and $\bar{y}:=\cP \bar{x}$. Then, we have $\cB \tilde{y}_n \to \bar{y}$ since 
\begin{align*}
|\cB \tilde{y}_n-\bar{y}|
& = |\cP(\cI-\cA)\cQ \tilde{y}_n - \cP \bar{x}|
= |\cP(\cI-\cA)\cQ  \cQ^{-1} \tilde{x}_n - \cP \bar{x}|
\\
& \leq  |\cP| |(\cI-\cA) \tilde{x}_n -\bar{x}|
\to 0 \text{ as } n \to \infty.
\end{align*}
In particular, for $i\in \{1,2,\cdots,r\}$,
\begin{equation}
\label{B2}
|\tilde{y}_n^i + (a_{i1} \tilde{y}_{n}^{r+1} + \cdots + a_{id-r} \tilde{y}_{n}^{d}) - \bar{y}^i| \to 0 \text{ as } n \to \infty, 
\end{equation}
where $z^j$ denotes the $j$-th component of $z \in \R^d$.
We take $\{y_n\} \subset \R^d$ satisfying the following properties.
\begin{equation*}
\l\{
\begin{array}{l}
y_n^i:= -(a_{i1} \tilde{y}_{n}^{r+1} + \cdots + a_{id-r} \tilde{y}_{n}^{d}),
\\
y_n^j := \tilde{y}_n^j,
\end{array}
\r.
\end{equation*}
for $i\in \{1,2,\cdots,r\}$ and $j\in\{r+1,r+2,\cdots,d\}$. Then, we have $\cB y_n =0$ for all $n \in \N$ by the definition of $\{y_n\}$. Moreover, we have for $i\in \{1,2,\cdots,r\}$ and $j\in\{r+1,r+2,\cdots,d\}$, 
\begin{align*}
\tilde{y}_n^i - y_n^i
&=\tilde{y}_n^i + (a_{i1} \tilde{y}_{n}^{r+1} + \cdots + a_{id-r} \tilde{y}_{n}^{d}) 
\underbrace{- (a_{i1} \tilde{y}_{n}^{r+1} + \cdots + a_{id-r} \tilde{y}_{n}^{d})  - y_n^i}_{=0 \text{ by the definition of $y_n^i$ }}
\\
&\to \bar{y}^i \text{ as } n \to \infty.
\end{align*}
and $\tilde{y}_{n}^{j}- y _n^j=0$. Therefore,
\[ \begin{array}{r} d\{ \\ kd-d \{\end{array} \begin{pmatrix} \tilde{y}_n \\ 0 \end{pmatrix} - \begin{pmatrix} y_n \\ 0 \end{pmatrix} \to \bar{y} \text{ as } n \to \infty. \]
Note that $\bar{y}^{r+1}=\bar{y}^{r+2}=\cdots=\bar{y}^{kd}=0$. Define $x_n:=\cQ y_n$ by \eqref[B1]. Then, we have
\[0=\cB y_n = \cP(\cI-\cA)\cQ y_n = \cP(\cI-\cA) x_n.\]
Multiplying $\cP^{-1}$ from the left, we obtain $\cA x_n = \cI x_n$. Moreover, 
\[ \tilde{x}_n - x_n= \cQ (\tilde{y}_n - y_n) \to \cQ \begin{pmatrix} \bar{y}^1 \\ \vdots \\ \bar{y}^d \end{pmatrix}=: x_\infty. \]
This completes the proof. 
\end{proof}

\begin{lemma}
\label{lemA.3}
Let $f\in H_{G'}^1$ and $\{x_n\}$ satisfy $|x_n - \cG x_n| \to \infty$ as $n\to \infty$ for $\cG \in G\setminus G'$.
We have the following identities.
\begin{align}
\label{A.3}
\norm[ \sum_{\cG \in G} \cG (\tau_{x_n} f )]_{\dot{H}^{\lambda}}^2
&=\frac{\#G}{\#G'} \norm[ \#G' f]_{\dot{H}^{\lambda}}^2 +o(1),
\\ 
\label{A.4}
\norm[ \sum_{\cG \in G} \cG (\tau_{x_n} f )]_{L^{p}}^{p}
&=\frac{\#G}{\#G'} \norm[ \#G' f]_{L^{p}}^{p} +o(1)
\end{align}
where $\lambda \in[0,1]$, $p\geq 1$ and $o(1)\to 0$ as $n\to \infty$. In particular, the following identity holds for any $\omega>0$.
\[ S_{\omega}\l(\sum_{\cG \in G} \cG (\tau_{x_n} f )\r)= \frac{\#G}{\#G'} S_{\omega}\l( \#G' f\r)+o(1).\]
\end{lemma}

To prove \lemref[lemA.3], we need Refined Fatou's lemma. See {\cite{BL83} and \cite[Theorem 1.9]{LL01}}.

\begin{lemma}[Refined Fatou's lemma] \label{RF}
Let $\{f_n\} \subset L^p(\R^d)$ with $\limsup_{n \to \infty}\norm[f_n]_{L^p}<\infty$. If $f_n \to f$ almost everywhere, then we have
\[ \int_{\R^d} \l|  |f_n|^p -|f_n-f|^p -|f|^p \r|  dx  \to 0 \text{ as } n \to \infty. \]
In particular, $\norm[f_n]_{L^p}^p - \norm[f_n-f]_{L^p}^p \to \norm[f]_{L^p}^p$ as $n \to \infty$. 
\end{lemma}

\begin{proof}[Proof of {\lemref[lemA.3]}]
Let $\{\cG_{k}\}_{k=1}^{\#G/\#G'}$ be the set of left coset representatives.
First, we prove \eqref[A.3]. 
\begin{align*}
\norm[ \sum_{\cG \in G}  \cG (\tau_{x_n} f)]_{\dot{H}^{\lambda}}^2
&=(\#G')^2 \norm[ \sum_{k=1}^{\#G/\#G'}  \cG_{k} (\tau_{x_n}f )]_{\dot{H}^{\lambda}}^2
\\
&=(\#G')^2 \l\{ \sum_{k=1}^{\#G/\#G'}   \norm[ \cG_{k} (\tau_{x_n} f )]_{\dot{H}^{\lambda}}^2
+ \sum_{\substack{k,l=1\\ k\neq l}}^{\#G/\#G'}  \rbra[ \cG_{k} (\tau_{x_n} f),  \cG_{l} (\tau_{x_n} f )]_{\dot{H}^{\lambda}} \r\}
\\
&=  \frac{\#G}{\#G'}  \norm[ \#G' f]_{\dot{H}^{\lambda}}^2
\\ & \quad +(\#G')^2 \sum_{\substack{ k,l=1\\ k\neq l}}^{\#G/\#G'}  \rbra[\cG_{k} (\tau_{x_n} f ),  \cG_{l} (\tau_{x_n}f )]_{\dot{H}^{\lambda}}
\end{align*}
The second term tends to $0$ as $n \to \infty$ since $|x_n - \cG x_n| \to \infty$ as $n\to \infty$ for $\cG \in G\setminus G'$ and $\cG_k^{-1}\cG_l \not\in \cG'$ for $k\neq l$. This implies \eqref[A.3].  Next, we prove \eqref[A.4]. Without loss of generality, we may assume that $\cG_1=(0,\cI_{d})$.
\begin{align*}
\norm[ \sum_{\cG \in G} \cG (\tau_{x_n} f )]_{L^{p}}^{p}
& = (\#G')^{p} \norm[ \sum_{k=1}^{\#G/\#G'} \cG_{k} (\tau_{x_n} f )]_{L^{p}}^{p}
\\
& = (\#G')^{p} \norm[ \tau_{x_n} f + \sum_{k=2}^{\#G/\#G'} \cG_{k} (\tau_{x_n} f )]_{L^{p}}^{p}
\\
& = (\#G')^{p} \norm[ f + \sum_{k=2}^{\#G/\#G'}  \tau_{-x_n + \cG_{k}x_n  }  \cG_{k} f ]_{L^{p}}^{p}.
\end{align*}
We set $f_n^1:= f + \sum_{k=2}^{\#G/\#G'}  \tau_{-x_n + \cG_{k}x_n  }  \cG_{k} f$. Then, $\sup_{n\in \N} \norm[f_n^1]_{L^p}<\infty$ and $f_n^1 \to f$ almost everywhere since $|x_n - \cG x_n| \to \infty$ as $n\to \infty$ for $\cG \in G\setminus G'$. Therefore, Refined Fatou's lemma gives us 
\[ \norm[f_{n}^1 ]_{L^{p}}^{p} - \norm[ f_n^1-f ]_{L^{p}}^{p} -\norm[ f  ]_{L^{p}}^{p} \to 0.\]
Here, $f_n^1-f= \sum_{k=2}^{\#G/\#G'}  \tau_{-x_n + \cG_{k}x_n  }  \cG_{k} f $ and 
\begin{align*} 
\norm[ \sum_{k=2}^{\#G/\#G'}  \tau_{-x_n + \cG_{k}x_n  }  \cG_{k} f ]_{L^p}^p
& =\norm[ \sum_{k=2}^{\#G/\#G'}  \tau_{\cG_{k}x_n  }  \cG_{k} f ]_{L^p}^p
\\ 
&= \norm[  \tau_{\cG_{2}x_n  }  \cG_{2} f  + \sum_{k=3}^{\#G/\#G'}  \tau_{\cG_{k}x_n  }  \cG_{k} f ]_{L^p}^p
\\ 
&= \norm[ \cG_{2}( \tau_{x_n  }  f)  + \sum_{k=3}^{\#G/\#G'}  \tau_{\cG_{k}x_n  }  \cG_{k} f ]_{L^p}^p
\\ 
&= \norm[  f  + \sum_{k=3}^{\#G/\#G'}\tau_{ -x_n + \cG_{2}^{-1} \cG_{k}x_n  }   \cG_{2}^{-1}  \cG_{k} f ]_{L^p}^p
\end{align*}
Let $f_n^2:=f  + \sum_{k=3}^{\#G/\#G'}\tau_{ -x_n+ \cG_{2}^{-1} \cG_{k}x_n  }   \cG_{2}^{-1}  \cG_{k} f $. Then, $\sup_{n\in \N} \norm[f_n^2]_{L^p}<\infty$ and $f_n^2 \to f$ almost everywhere since $|x_n - \cG x_n| \to \infty$ as $n\to \infty$ for $\cG \in G\setminus G'$. Therefore, by Refined Fatou's lemma again, we get
\[ \norm[f_{n}^2 ]_{L^{p}}^{p} - \norm[ f_n^2-f ]_{L^{p}}^{p} -\norm[ f  ]_{L^{p}}^{p} \to 0.\]
By repeating this procedure, we get
\[ \norm[f_n^1]_{L^p}^p - \norm[f_n^{\#G/\#G'-1}-f]_{L^p}^p -( \#G/\#G' -1)\norm[f]_{L^p}^p \to 0,\]
where $f_n^{\#G/\#G'-1} :=f+\tau_{ -x_n+\cG_{\#G/\#G'-1}^{-1} \cdots  \cG_{2}^{-1} \cG_{\#G/\#G'}x_n  } \cG_{\#G/\#G'-1}^{-1} \cdots  \cG_{2}^{-1} \cG_{k} f$. Now, we have
\begin{align*}  
\norm[f_n^{\#G/\#G'-1}-f]_{L^p}^p 
&=\norm[\tau_{ -x_n+\cG_{\#G/\#G'-1}^{-1} \cdots  \cG_{2}^{-1} \cG_{\#G/\#G'}x_n  } \cG_{\#G/\#G'-1}^{-1} \cdots  \cG_{2}^{-1} \cG_{k} f]_{L^p}^p
\\
& =\norm[f]_{L^p}^p. 
\end{align*}
Thus, we obtain
\begin{align*}
\norm[ \sum_{\cG \in G} \cG (\tau_{x_n} f )]_{L^{p}}^{p}  
= (\#G')^{p} \norm[ f_n^1]_{L^{p}}^{p} 
\to (\#G')^{p}  \frac{\#G}{\#G'}\norm[f]_{L^p}^p.
\end{align*}
This completes the proof.
\end{proof}

\begin{lemma}
\label{lemA.2}
Let $\{t_n\}\subset \R$, $\{x_n\} \subset \R^d$, $\{f_n\} \subset H^1(\R^d)$ and $\psi \in H^1(\R^d) \setminus \{0\}$ satisfy 
\[ f_n \wto 0 \text{ and } e^{it_n \Delta} f_n (\cdot + x_n) \wto \psi \text{ as } n \to \infty \text{ in } H^1(\R^d).\]
Then, $|t_n|\to \infty$ or  $|x_n| \to \infty$ as $n \to \infty$ taking a subsequence.
\end{lemma}

\begin{proof}
If $|t_n| + |x_n|$ is bounded, then $f_n (\cdot + x_n) \wto 0$ since $f_n \wto 0 $. This contradicts $\psi \neq 0$.
See also \cite[Lemma 5.3]{FXC11}. 
\end{proof}

\begin{lemma}[{\cite[Proposition A.1]{BV16}}] \label{lemA.6}
For $j \in \{ 1,2\}$, let $V^j \in C(\R:H^1(\R^d)) \cap L^{\alpha}(\R:L^r(\R^d))$ and $\{(t_n^j,x_n^j)\}_{n\in\N}\subset \R \times \R^d$ satisfy $|t_n^1-t_n^2|\to \infty$ or $|x_n^1-x_n^2|\to \infty$ as $n \to \infty$. Then, we have
\[ \norm[ |V^1(\cdot - t_n^1,\cdot - x_n^1)|^{p-1}|V^2(\cdot - t_n^2,\cdot - x_n^2)|]_{L^{\beta'}(\R:L^{r'} )}\to 0.\]
\end{lemma}

\begin{proof}
First, we assume that $|t_n^1-t_n^2| \to \infty$. Then, since we have
\begin{align*}
&\norm[ |V^1(\cdot-t_n^1,\cdot-x_n^1)|^{p-1} |V^2(\cdot-t_n^2,\cdot-x_n^2)| ]_{L^{\beta'}(\R:L^{r'})} 
\\
&\leq \norm[ \norm[ V^1(\cdot-t_n^1)]_{L_x^r}^{p-1} \norm[ V^2(\cdot-t_n^2)]_{L_x^r} ]_{L_t^{\beta'}(\R)},
\end{align*}
we obtain the statement. Next we assume that $\{t_n^1-t_n^2\}$ is bounded and $|x_n^1-x_n^2|\to \infty$ as $n \to \infty$. 
We note that
\begin{align*}
&\norm[ |V^1(\cdot,\cdot-x_n^1)|^{p-1} |V^2(\cdot+t_n^1-t_n^2,\cdot-x_n^2)| ]_{L^{\beta'}(\{|t|>T\}:L^{r'})} 
\\
&\leq \norm[ V^1]_{L^{\alpha}(\{|t|>T\}:L^{r})} ^{p-1}
\norm[ V^2(\cdot+t_n^1-t_n^2) ]_{L^{\alpha}(\{|t|>T\}:L^{r})}  \to 0 \text{ as } T \to \infty.
\end{align*}
Thus, it suffices to prove that 
\begin{equation}  \label{A.5.0}
\norm[ |V^1(\cdot,\cdot-x_n^1)|^{p-1} |V^2(\cdot+t_n^1-t_n^2,\cdot-x_n^2)| ]_{L^{\beta'}(\{|t|\leq T\}:L^{r'})} \to 0 \text{ as } n \to \infty, 
\end{equation}
for any fixed $T>0$. We have, for any $t$, 
\begin{align*}  
&\norm[ |V^1(t,\cdot-x_n^1)|^{p-1} |V^2(t+t_n^1-t_n^2,\cdot-x_n^2)| ]_{L^{r'}}
\\
&= \norm[ |V^1(t)|^{p-1} |V^2(t+t_n^1-t_n^2,\cdot+x_n^1-x_n^2)| ]_{L^{r'}}
\end{align*}
By the Sobolev embedding, we have $|V^1(t)|^{p-1} \in L^{\frac{r}{p-1}}$ for any $t$. Moreover, we see that $\{V^2(t+t_n^1-t_n^2,x): x\in \R^d\}$ is compact in $L^r(\R^d)$ for any $t$ since $\{t_n^1-t_n^2\}$ is bounded. Thus, by $|x_n^1-x_n^2|\to \infty$, we see that
\[  \norm[ |V^1(t)|^{p-1} |V^2(t+t_n^1-t_n^2,\cdot+x_n^1-x_n^2)| ]_{L^{r'}} \to 0 \text{ as } n \to \infty \text{ for any } t. \]
On the other hand, since $\{t_n^1-t_n^2\}$ is bounded and $V^j$ belongs to $C(\R:H^1(\R^d))$ for $j\in\{1,2\}$, the Sobolev embbeding gives us that
\begin{align*}
\sup_{t \in [-T,T]}  & \norm[ |V^1(t)|^{p-1} |V^2(t+t_n^1-t_n^2,\cdot+x_n^1-x_n^2)| ]_{L^{r'}}
\\
& \leq \sup_{t \in [-T,T]}  \l( \norm[ V^1(t)]_{L^{r}}  \norm[ V^2(t+t_n^1-t_n^2)]_{L^{r}}\r)
\\
& \cleq \sup_{t \in [-T,T]}  \l( \norm[ V^1(t)]_{H^1}  \norm[ V^2(t+t_n^1-t_n^2)]_{H^1}\r)
\\
&<C<\infty.
\end{align*}
Therefore, by the Lebesgue dominated convergence theorem, we obtain \eqref[A.5.0].
\end{proof}

\begin{lemma}[{\cite[Corollary A.2]{BV16}, \cite[Lemma 5.5]{IMN11}}]
\label{lemA.4}
For $j \in \{ 1,2,\cdots, J\}$, let $V^j \in C(\R:H^1(\R^d)) \cap L^{\alpha}(\R:L^r(\R^d))$ and $\{(t_n^j,x_n^j)\}_{n\in\N}\subset \R \times \R^d$ satisfy $|t_n^j-t_n^h|\to \infty$ or $|x_n^j-x_n^h|\to \infty$ as $n \to \infty$ if $j\neq h$. Then, we have
\[ \limsup_{n \to \infty} \norm[\sum_{j=1}^{J} V_n^j]_{L^{\alpha}(\R:L^r )} \leq 2 \sum_{j=1}^{J}  \norm[V^j]_{L^{\alpha}(\R:L^r )},\]
where $V_n^j(t,x):=V^j(t+t_n^j,x-x_n^j)$.
\end{lemma}

\begin{proof}
We have
\begin{align*} 
\norm[\sum_{j=1}^{J} V_n^j]_{L^{\alpha}(\R:L^r )} 
&= \norm[ \l|\sum_{j=1}^{J} V_n^j\r|^{p} ]_{L^{\beta'}(\R:L^{r'} )}^{\frac{1}{p}}
\\
& \leq \l( \norm[ \l|\sum_{j=1}^{J} V_n^j\r|^{p} - \sum_{j=1}^{J} |V_n^j|^{p}  ]_{L^{\beta'}(\R:L^{r'} )}+  \norm[ \sum_{j=1}^{J} |V_n^j|^{p}  ]_{L^{\beta'}(\R:L^{r'} )} \r)^{\frac{1}{p}}
\\
& \leq 2^{\frac{1}{p}} \norm[ \l|\sum_{j=1}^{J} V_n^j\r|^{p} - \sum_{j=1}^{J} |V_n^j|^{p}  ]_{L^{\beta'}(\R:L^{r'} )}^{\frac{1}{p}} + 2^{\frac{1}{p}} \sum_{j=1}^{J}  \norm[ V_n^j ]_{L^{\alpha}(\R:L^{r} )} 
\end{align*}
By the inequality
\[ \l| \l|  \sum_{j=1}^J z_j \r|^{p} - \sum_{j=1}^J |z_j|^{p} \r| \leq C \sum_{j\neq k} |z_j||z_k|^{p-1},\]
and \lemref[lemA.6], we obtain the statement. 
\end{proof}

\section{Applications}
\label{secB}

We introduce some applications of \thmref[thm1.1] in this appendix. Here, we only treat examples in one and two dimensional cases.

\noindent\underline{In the one dimensional case.}
We have only three subgroups of $\R/2\pi\Z \times O(1)$ satisfying the assumption (A). Namely, we have
\begin{align*}
G_0&:=\{(0,1)\},
\\
G_{even}&:=\{ (0,1),(0,-1)\},
\\
G_{odd}&:=\{ (0,1), (\pi,-1)\}.
\end{align*}
When $G=G_0$, we can classify the solutions with $S_{\omega}<l_{\omega}$ into scattering and blow-up by the functional $K$ by \thmref[thm1.1.0] (\cite{FXC11, AN13}). Noting that $Q_{\omega}$ is radially symmetric, we can classify the even solutions (\textit{i.e.} in the case of $G=G_{even}$) with $S_{\omega}<l_{\omega}$.
On the other hand, by \thmref[thm1.1], we can classify the odd solutions (\textit{i.e.} $G=G_{odd}$) with 
\[ S_{\omega} < \min\{ 2l_{\omega}, l_{\omega}^{G_{odd}} \}=2l_{\omega}, \]
where we have used $\bS_{\omega}^{G_0}=l_{\omega}$ and the fact that $l_{\omega}^{G_{odd}}=2l_{\omega}$.
This means that we can classify the solutions above the ground state standing waves by oddness. 

\noindent\underline{In the two dimensional case.}
Unlike the one dimensional case, we have many subgroups in the two dimensional case. Here, we only introduce three applications.  

\noindent(1). We consider 
\[ G=\l\{ 
\l(0, \begin{pmatrix} 1 & 0 \\ 0 & 1 \end{pmatrix} \r), 
\l(\pi, \begin{pmatrix} 0 & -1 \\ 1 & 0 \end{pmatrix} \r),
\l(\pi, \begin{pmatrix} 0 & 1 \\ -1 & 0 \end{pmatrix} \r),
\l(0, \begin{pmatrix} -1 & 0 \\ 0 & -1 \end{pmatrix} \r)
\r\}.\]
The proper subgroups of $G$ are as follows. 
\begin{align*}
G_0
&:=\l\{ 
\l(0, \begin{pmatrix} 1 & 0 \\ 0 & 1 \end{pmatrix} \r)
\r\},
\\
G_1
&:=\l\{ 
\l(0, \begin{pmatrix} 1 & 0 \\ 0 & 1 \end{pmatrix} \r), 
\l(0, \begin{pmatrix} -1 & 0 \\ 0 & -1 \end{pmatrix} \r)
\r\}.
\end{align*}
By \thmref[thm1.1], we classify the solutions with $G$-invariance satisfying
\[ S_{\omega} < \min\{ 4l_{\omega}, l_{\omega}^{G} \},\]
where we note that $G_1$ does not satisfy the condition {\rm ($\ast$)}. 
Since $l_{\omega}^{G} >l_{\omega}$, we can classify the solutions above the ground state standing waves by the group invariance. 

\noindent(2). We consider 
\[ G=\l\{ 
\l(0, \begin{pmatrix} 1 & 0 \\ 0 & 1 \end{pmatrix} \r), 
\l(\pi, \begin{pmatrix} -1 & 0 \\ 0 & 1 \end{pmatrix} \r)
\r\}.\]
By \thmref[thm1.1], we classify the solutions with $G$-invariance satisfying
\[ S_{\omega} < \min\{ 2l_{\omega}, l_{\omega}^{G} \}=2l_{\omega},\] 
where we have used the fact that $l_{\omega}^{G}=2l_{\omega}$. 

\noindent(3). We consider 
\[ G=\l\{ 
\l(\theta, \begin{pmatrix} \cos \theta & -\sin \theta \\ \sin\theta & \cos \theta \end{pmatrix} \r):
\theta \in [0,2\pi)
\r\}.\]
Then, the embedding $H_{G}^1 \hookrightarrow L^{p+1}(\R^d)$ is compact since $Gx$ has infinitely many elements for $x \neq 0$. By \thmref[thm1.1], we can classify the solutions such that $S_{\omega}< l_{\omega}^{G}$.


\section*{Acknowledgement}
 The authors would like to express deep appreciation to Professor Kenji Nakanishi and Doctor Masahiro Ikeda for many useful suggestions, comments and constant encouragement. The author  is supported by Grant-in-Aid for JSPS Research Fellow 15J02570.


\end{document}